\documentclass[11pt,oneside,english]{amsart}

\usepackage[T1]{fontenc}
\usepackage[latin9]{inputenc}
\usepackage{geometry}
\usepackage{babel}
\usepackage{amsthm}
\usepackage{amssymb}
\usepackage[babel=true]{csquotes}
\usepackage{slashbox}
\usepackage{ulem}
\usepackage[unicode=true,pdfusetitle,
bookmarks=true,bookmarksnumbered=false,bookmarksopen=false,
 breaklinks=false,pdfborder={0 0 1},backref=false,colorlinks=false]
 {hyperref}

\newtheorem{thm}{Theorem}[section]

\newtheorem{lem}[thm]{Lemma}
\newtheorem{prop}[thm]{Proposition}
\newtheorem{cor}[thm]{Corollary}
\theoremstyle{definition}
\newtheorem{definition}[thm]{Definition}
\newtheorem{remark}[thm]{Remark}
\newtheorem{example}[thm]{Example}

\newcommand{\ra}{\rightarrow}

\def\g{\mathfrak{g}}

\newcommand{\bbB}{\mathbf}

\newcommand{\NN}{{\bbB N}}

\newcommand{\QQ}{{\bbB Q}}
\newcommand{\RR}{{\bbB R}}
\newcommand{\ZZ}{{\bbB Z}}

\newcommand{\bg}{{\bbB g}}
\newcommand{\bX}{{\bbB X}}

\newcommand{\cF}{{\mathcal F}}

\newcommand{\cL}{{\mathcal L}}
\newcommand{\cM}{{\mathcal M}}

\newcommand{\bx}{{\mathbf x}}
\newcommand{\bN}{{\mathbb N}}

\newcommand{\bR}{{\mathbb R}}

\newcommand{\fg}{{\mathfrak g}}

\newcommand{\sx}{{\mathsf x}}
\newcommand{\sr}{{\mathsf r}}

\newcommand{\su}{{\mathsf u}}

\DeclareMathOperator{\Span}{span}
\DeclareMathOperator{\ad}{ad}
\DeclareMathOperator{\Lie}{Lie}

\DeclareMathOperator{\End}{End}
\DeclareMathOperator{\diag}{diag}
\DeclareMathOperator{\rrk}{rk}
\DeclareMathOperator{\codim}{codim}
\DeclareMathOperator{\maj}{maj}

\makeatother

\begin{document}

\title{Diophantine properties of nilpotent Lie groups}

\begin{abstract}
A finitely generated subgroup $\Gamma$ of a real Lie group $G$ is said to be Diophantine if there is $\beta>0$ such that non-trivial elements in the word ball $B_\Gamma(n)$ centered at $1 \in \Gamma$ never approach the identity of $G$ closer than $|B_\Gamma(n)|^{-\beta}$. A Lie group $G$ is said to be Diophantine if for every $k \geq 1$ a random $k$-tuple in $G$ generates a Diophantine subgroup. Semi-simple Lie groups are conjectured to be Diophantine but very little is proven in this direction. We give a characterization of Diophantine nilpotent Lie groups in terms of the ideal of laws of their Lie algebra. In particular we show that nilpotent Lie groups of class at most $5$, or derived length at most $2$, as well as rational nilpotent Lie groups are Diophantine. We also find that there are non Diophantine nilpotent and solvable (non nilpotent) Lie groups.  \end{abstract}

\author{Menny Aka,  Emmanuel Breuillard, Lior Rosenzweig, Nicolas de Saxc\'e}



\address{Section de math�matiques\\ EPFL\\
Station 8 - B\^at. MA\\
CH-1015 Lausanne\\
SWITZERLAND}
\email{menashe-hai.akkaginosar@epfl.ch}

\address{Laboratoire de Math\'ematiques\\
B\^atiment 425, Universit\'e Paris Sud 11\\
91405 Orsay\\
FRANCE}
\email{emmanuel.breuillard@math.u-psud.fr}

\address{Department of mathematics\\
KTH\\
SE-100 44  Stockholm\\
SWEDEN}
\email{lior.rosenzweig@gmail.com}

\address{Einstein Institute of Mathematics\\
Givat Ram\\
The Hebrew University\\
Jerusalem, 91904\\
ISRAEL}
\email{saxce@ma.huji.ac.il}

\maketitle
\setcounter{tocdepth}{1}

\tableofcontents

\section{Introduction}

Let $G$ be a connected real Lie group, endowed with a left-invariant Riemannian metric $d$. We investigate Diophantine properties of finitely generated subgroups of $G$. If $\Gamma=\langle S\rangle$ is a subgroup of $G$ with finite generating set $S$, we define, for $n\in\bN$,
$$\delta_\Gamma(n)=\min\{ d(x,1)\,|\,x\in B_\Gamma(n),\, x\neq 1\}$$
where $B_\Gamma(n)=(S  \cup S^{-1} \cup\{1\})^n$ is the ball of radius $n$ centered at $1$ for the word metric defined by the generating set $S$.
We will say that $\Gamma=\langle S\rangle$ is \emph{$\beta$-Diophantine} if there exists a constant $c>0$ (depending maybe on $\Gamma$ and $S$ and $d$) such that for all $n\geq 1$,
$$\delta_\Gamma(n)\geq c\cdot|B_{\Gamma}(n)|^{-\beta}.$$
 In general, the Diophantine exponent $\beta$ may  depend on the choice of a generating set $S$. But if $G$ is nilpotent, then it depends neither on the choice of generating set of $\Gamma$, nor on the choice of left-invariant Riemannian metric. So the notion of a $\beta$-Diophantine finitely generated subgroup of $G$ makes sense without mention of a choice of generating set, or metric\footnote{The left-invariant assumption is just for convenience, because any two Riemannian metrics on $G$ are comparable in a neighborhood of the identity; however it is important that the metric be Riemannian. Allowing the metric to be sub-Riemannian could affect the Diophantine exponents, yet not the property of being Diophantine.}. This notion naturally generalizes the classical notion of a Diophantine number in dynamical systems and number theory: $\gamma \in \RR/\ZZ$ is a Diophantine rotation if and only if the cyclic subgroup it generates in the torus $\RR/\ZZ$ is Diophantine in our sense.

We will say that the ambient group $G$ is \emph{$\beta$-Diophantine} for $k$-tuples if almost every $k$-tuple of elements of $G$, chosen with respect to the Haar measure of $G^k$, generates a $\beta$-Diophantine subgroup. We will also say that $G$ is Diophantine for $k$-tuples if $G$ is $\beta$-Diophantine for $k$-tuples and some $\beta=\beta(k)>0$. Finally we say that $G$ is Diophantine if for every integer $k \geq 1$ it is Diophantine for $k$-tuples. We will show that for every $k$ there are Lie groups $G$ that are Diophantine for $k$-tuples, but not for $(k+1)$-tuples (see Theorem \ref{everykk}). However if $G$ is Diophantine and nilpotent then $\beta$ can always be chosen independently of $k$ (see Theorem \ref{concludeth}).

\bigskip

When the ambient group $G$ is a connected abelian Lie group, the behavior of $\delta_\Gamma(n)$ depends on the rational approximations to a set of generators of $\Gamma$. This case has been much studied in the past and constitutes a classical chapter of Diophantine approximation. For example, from Dodson's generalization \cite{dodson} of Jarn\'ik's Theorem \cite{jarnik}, it is not difficult to compute the Hausdorff dimension of the set of $\beta$-Diophantine $k$-tuples in any abelian connected Lie group.

On the other hand, little is known in the non-commutative setting. Gamburd Jakobson and Sarnak \cite{GJS} first identified the relevance of the Diophantine property for dense subgroups of $G=SU(2)$ in their study of the spectral properties of averaging operators on $L^2(SU(2))$ and the speed of equidistribution of random walks therein. In particular they stressed that although there is a $G_\delta$-dense set of $k$-tuples in $SU(2)$ which are not Diophantine (this assertion still holds in any nilpotent Lie group as we will see cf. Prop. \ref{baire}), every $k$-tuple whose matrix entries are algebraic numbers is a Diophantine tuple. Then it is natural to conjecture (as Gamburd et al. did in \cite[Conjecture 29]{gamburd-hoory}) that almost every $k$-tuple in the measure-theoretic sense is Diophantine. Although this remains an open problem, Kaloshin and Rodnianski \cite{KR} have shown that almost every $k$-tuple in $SU(2)$ generates a subgroup $\Gamma$ which is weakly Diophantine in the sense that $B_\Gamma(n)$ avoids a ball of radius $e^{-C_kn^2}$ centered at $1$ (the genuine Diophantine property as defined above would require this radius to be much larger, namely bounded below by some $e^{-C_kn}$). Similarly the second named author showed in \cite[Corollary 1.11]{breuillard-sl2} (as a consequence of the uniform Tits alternative) that every $k$-tuple generating a dense subgroup of $SU(2)$ satisfies a closely related weak form of the Diophantine property. Finally Bourgain and Gamburd proved in \cite{bourgaingamburdsu2} that every (topologically generating) $k$-tuple in $SU(2)$ whose matrix entries are algebraic numbers has a spectral gap in $L^2(SU(2))$. Their proof uses the Diophantine property of these $k$-tuples in an essential way. Conjecturally almost all, and perhaps even all, topologically generating $k$-tuples in $SU(2)$ have a spectral gap (cf. Sarnak's spectral gap conjecture \cite[p. 58]{sarnakmodularforms}).

Recently P. Varj\'u \cite{varjudiophantine} tackled a similar problem for the group of affine transformations of the line, the $\{ax+b\}$ group, and showed that in a certain one-parameter family of $2$-tuples in this group, almost every $2$-tuple is Diophantine.

\bigskip

In this paper we investigate Diophantine properties of finitely generated subgroups of nilpotent and solvable Lie groups. Perhaps the main surprise in our findings is that, in sharp contrast with the abelian case,  not every solvable or nilpotent real Lie group is Diophantine. We exhibit examples of connected nilpotent Lie groups $G$ of nilpotency class $6$ and higher, and of non-nilpotent connected solvable Lie groups $G$ of derived length $3$ and higher, which are not Diophantine. And indeed in those examples one can find, for every $k \geq 3$, a sequence of words $w_n$ in $k$-letters, which are not laws of $G$, but behave like almost laws (see \S \ref{almostlaws}) inasmuch as for any fixed compact set $K$ in $G^k$, $w_n(K) \to 1$ with arbitrarily fast speed as the length of the word $l(w_n)$ tends to $+\infty$. To put it briefly, the bad behavior of Liouville numbers with respect to Diophantine approximation by rationals can be replicated to build real Lie algebras exhibiting this bad behavior. Of course such Lie algebras are not defined over $\QQ$.

\bigskip

If $\fg$ is a Lie algebra of step $s$, we define the ideal of laws on $k$ letters $\cL_{k,s}(\fg)$ of $\fg$ as the set of all elements of the free step $s$ nilpotent Lie algebra on $k$ generators $\cF_{k,s}$ that are identically zero when evaluated on $k$-tuples of elements of $\fg$.
We show that the property of being Diophantine for a connected $s$-step nilpotent Lie group $G$ is intimately related to the way the  ideal of laws on $k$ letters $\cL_{k,s}(\fg)$ of the Lie algebra $\fg$ of $G$ sits in the free $s$-step nilpotent Lie algebra $\cF_{k,s}$ on $k$ generators. Although $\cF_{k,s}$ is naturally defined over $\QQ$, the ideal of laws $\cL_{k,s}(\fg)$ may not be and we have a natural epimorphism of real Lie algebras:

$$\Lambda: \cF_{k,s} / \cL_{k,s}(\fg)_{\QQ} \to \cF_{k,s} / \cL_{k,s}(\fg)$$
where $\cL_{k,s}(\fg)_{\QQ}$ is the real vector space spanned by $\cL_{k,s}(\fg) \cap\cF_{k,s}(\QQ)$. It is the Diophantine properties of the kernel of $\Lambda$ viewed as a real subspace of $\cF_{k,s} / \cL_{k,s}(\fg)_{\QQ}$ which determines the Diophantine properties of $G$ as a Lie group. In brief, a real subspace of a real vector space defined over $\QQ$ is said to be Diophantine if integer points outside it cannot be too close to it (see Definition \ref{diophantine-subspace} below). It turns out that if $s\leq 5$, then $\cL_{k,s}(\fg)$ is always defined over $\QQ$ regardless of $\fg$, and hence Diophantine. Recall that there is a continuous family of non-isomorphic $2$-step nilpotent Lie algebras, so $\fg$ may not be defined over $\QQ$ or over a number field and still give rise to a Diophantine Lie group. Similarly if $G$ is metabelian (i.e. if $\cL_{k,s}$ contains the ideal generated by all double commutators $[[x,y],[z,w]]$), then  $\cL_{k,s}(\fg)$ is again defined over $\QQ$, regardless of $s$ and of $\fg$ (see Theorem \ref{metabelian} below). This follows from a study (made in the appendix) of the multiplicities of the irreducible $SL_k$-submodules of $\cL_{k,s}(\fg)$, where $SL_k$ acts by substitution of the variables.

\bigskip

We now summarize the above discussion and state our results. See Section~\ref{sec:characterization} for precise definitions and notation. The nilpotent real Lie groups considered in this paper will always be endowed with a left-invariant Riemannian metric.

\begin{thm}\label{characterizationi}
Let $G$ be a connected nilpotent Lie group with Lie algebra $\fg$. For any integer $k\geq 1$, the following are equivalent.
\begin{enumerate}
\item $G$ is Diophantine for $k$-tuples.
 \item The ideal of laws on $k$ letters of $\fg$, $\cL_{k,s}(\fg)$, is a Diophantine subspace of the free $s$-step nilpotent Lie algebra on $k$ generators $\cF_{k,s}$.
\end{enumerate}
\end{thm}

It may seem surprising at first glance that this criterion for diophantinity is given only in terms of Lie algebra laws as opposed to group laws. However, as is well-known, nilpotent Lie groups are best analyzed through their Lie algebra and there is a simple relation between word maps in $G$ and bracket maps in $\fg$ (see Lemma \ref{lawtoword} below).

The proof of Theorem \ref{characterizationi} relies on a Borel-Cantelli argument (analogous to the classical one showing that almost every real number is Diophantine and to the one used by Kaloshin-Rodnianski \cite{KR}) and sub-level sets estimates for polynomial maps, originating in the work of Remez (see Theorem \ref{prop:estimates on Polynomial maps}) and related to the $(C,\alpha)$-good functions of Kleinbock and Margulis \cite{kleinbock-margulis}.

This criterion immediately settles the case of nilpotent Lie groups whose Lie algebra has rational (or even algebraic) structure constants:

\begin{cor}\label{rationali}
If $G$ is a connected rational nilpotent Lie group, then it is Diophantine.
\end{cor}

Combining Theorem~\ref{characterizationi} with a study of the structure of fully invariant ideals of the free Lie algebra on $k$ letters, we are also able to construct examples of non-Diophantine nilpotent Lie groups, and to show that such groups must have nilpotency step at least $6$:

\begin{thm}\label{nondiophantinei}
Fix an integer $k\geq 3$ (resp. $k=2$). Then, all connected nilpotent Lie groups of step $s\leq 5$ (resp. $s\leq 6$) are Diophantine for $k$-tuples, but for all integers $s>5$ (resp. $s>6$), there exists a connected nilpotent Lie group of step $s$ that is not Diophantine for $k$-tuples.
\end{thm}

The numbers $s=6$ and $s=5$ appear here for the following reason: the free $s$-step nilpotent Lie algebra on $k$ letters  $\cF_{k,s}$ admits multiplicity as an $SL_k$-module if and only if $s \geq 6$ when $k \geq 3$ and if and only if $s \geq 7$ when $k=2$. See Theorem \ref{multfree}. The existence of multiplicity allows to construct fully invariant ideals of $\cF_{k,s}$ which are not Diophantine, hence giving rise to non-Diophantine Lie groups via Theorem~\ref{characterizationi}.

The decomposition of $\cF_{k,s}$ into irreducible $SL_k$-modules has been very much studied in the literature on free Lie algebras starting with Witt \cite{witt} and Klyachko \cite{klyachko}. See  \cite{reutenauer} for a book treatment. We survey some of these results in the appendix. In particular, we recall a relatively recent formula due to Kraskiewicz and Weyman \cite{kraskiewiczweyman}, which allows one to compute the multiplicity of any given irreducible submodule of $\cF_{k,s}$ in terms of its Young diagram. Together with Theorem \ref{characterizationi} this allows us to show the following:

\begin{thm}\label{everykk}
For every integer $k \geq 1$, there exists a connected (nilpotent) Lie group, which is Diophantine for $k$-tuples, but not for $(k+1)$-tuples.
\end{thm}

We also give two other applications of Theorem~\ref{characterizationi}. The first one concerns connected nilpotent Lie groups that are metabelian, i.e. solvable of derived length $2$:

\begin{thm}\label{metabeliani}
Every metabelian connected nilpotent Lie group is Diophantine.
\end{thm}

\noindent Again the point here is the fact that the free metabelian Lie algebra is multiplicity-free as an $SL_k$-module (see Lemma \ref{irreducible}). The second application is a construction of a solvable non-nilpotent non-Diophantine Lie group:

\begin{prop}
There exists a connected solvable non-nilpotent Lie group that is not Diophantine.
\end{prop}



For a nilpotent Lie group with Lie algebra $\fg$ set $d_i=\dim \fg^{(i)}/\fg^{(i+1)}$, where $\fg^{(i)}$ is the $i$-th term of the central descending series. The next statement clarifies what happens to a topologically generic tuple in Diophantine and non Diophantine nilpotent Lie groups.

\begin{thm}\label{nondiophtuples}Let $G$ be a connected $s$-step nilpotent Lie group.
\begin{enumerate}
\item Regardless of whether $G$ is Diophantine or not, if $k>d_1$, then there is a $G_\delta$-dense set of $k$-tuples in $G^k$ which are not Diophantine.
\item If $G$ is non-Diophantine for $k$-tuples, then there is a proper closed algebraic subvariety of $G^k$ outside of which every $k$-tuple is non-Diophantine. In particular there is an open dense subset of $G^k$ made of non-Diophantine tuples.
\end{enumerate}
\end{thm}

Finally we study the dependence in $k$, the size of the $k$-tuple and prove the following (see \S \ref{dependence}). 

\begin{thm}\label{concludeth} Let $G$ be a connected $s$-step nilpotent Lie group.
\begin{enumerate}
\item  $G$ is Diophantine if and only if it is Diophantine for $s$-tuples.
\item If $G$ is Diophantine, and $\beta_k$ is the infimum of the $\beta>0$ such that $G$ is $\beta$-Diophantine on $k$-tuples, $\frac{1}{d_s} \leq \liminf \beta_k \leq \limsup \beta_k \leq s$ as $k \to +\infty$.
\end{enumerate}
\end{thm}

The lower bound on $\liminf \beta_k$ follows from a simple pigeonhole argument (analogous to Dirichlet's principle in Diophantine approximation), while the upper-bound on $\limsup \beta_k$ relies on Theorem \ref{characterizationi} and the sharp Remez-type inequality derived in Theorem \ref{prop:estimates on Polynomial maps} together with a study of $\cL_{k,s}(\fg)$ as $k$ grows.

\begin{remark}From this theorem, we see that the exponent $\beta$ seems to carry some interesting information on the group $G$. It is reassuring that it does not degenerate to either $0$ or $+\infty$ as $k$ grows. More importantly it makes much sense to try to determine the critical $\beta>0$ below which almost no $k$-tuple is Diophantine and above which almost all are. The work of Kleinbock and Margulis \cite{kleinbock-margulis} is much relevant here. For example in the case of the Heisenberg groups, a simple application of \cite{kleinbock-margulis} yields the exact value of critical $\beta$ for each $k$. We plan to tackle this problem in a future paper.
\end{remark}

\bigskip

The paper is organized as follows. After giving the general setting of our work and some preliminary lemmas in Section~\ref{sec:preliminaries}, we prove Theorem~\ref{characterizationi} in Section~\ref{sec:characterization}. Section~\ref{sec:submodules} is devoted to the study of fully invariant ideals of the free Lie algebra and to the proof of Theorem~\ref{nondiophantinei}. This section makes use of some computations on the free Lie algebra given in the appendix. The last section contains the proofs of Theorems \ref{everykk} \ref{nondiophtuples} and \ref{concludeth} and further remarks. 

\bigskip

\noindent \emph{Acknowledgements.} We would like to thank D. Segal and Y. de Cornulier for providing useful references about group varieties. We are also grateful to Y. Benoist, A. Gamburd,  A. Gorodnik, N. Monod and P. Varj\'u for interesting discussions regarding various issues related to this paper. The referee's comments and careful reading were also of great help and we are happy to thank him or her. E.B. and N.S. acknowledge partial support from the European Research Council (ERC) Grant GADA 208091, ERC AdG Grant 267259, ANR-12-BS01-0011 and ANR-13-BS01-0006-01, while M.A. acknowledges the support of ISEF and Advanced Research Grant 228304 from the ERC. L.R. was supported by the G\"oran Gustafsson Foundation (KVA).

\section{Preliminaries}
\label{sec:preliminaries}

\subsection{Nilpotent Lie groups}

We briefly review some elementary theory of nilpotent Lie groups.
Proofs of all the results mentioned below may be found in \cite[Chapter 1]{CG90}.
Let $\fg$ be a Lie algebra over $\bR$. The \emph{descending central series}
of $\fg$ is defined inductively by
\[
\fg^{(1)}:=\fg,\fg^{(l+1)}:=[\fg,\fg^{(l)}]=\Span_\mathbb R\{[X,Y]: X\in\fg,Y\in\fg^{(l)}\}
\]
The algebra $\fg$ is said to be \emph{nilpotent }if $\fg^{(s+1)}=0$
for some $s\in\mathbb N$. If $s$ is the smallest integer such that, $\fg^{(s+1)}=0$, then $\fg$ is
said to be \emph{nilpotent of step (or class) $s$}.

A connected Lie group $G$ is called nilpotent if its Lie algebra is nilpotent, or equivalently, if it is nilpotent as a group. We
have the following basic result
(see \cite[Theorem 1.2.1]{CG90}).
\begin{prop}
\label{prop:exp Sur}Let $G$ be a connected simply connected nilpotent
Lie group with Lie algebra $\fg.$ Then, the exponential map $exp:\fg\ra G$
is a diffeomorphism.
\end{prop}

As a consequence, every connected simply connected nilpotent
Lie group $G$ can be diffeomorphically identified with its Lie algebra
$\fg$.
The product operation on $G\simeq (\fg,*)$ is then given explicitly by the
Campbell-Baker-Hausdorff formula:\\
For all $X,Y\in\fg$,
$$X*Y=\sum_{n>0}\frac{(-1)^{n+1}}{n}\sum_{\substack{p_i+q_i>0\\1\leq i\leq n}}\frac{(\sum_{1\leq i\leq n} p_i+q_i)^{-1}}{p_1!q_1!\cdots p_n!q_n!}(\ad X)^{p_1}(\ad Y)^{q_1}\dots (\ad X)^{p_n}(\ad Y)^{q_n-1}Y.$$
(If $q_n=0$, the term in the sum ends with $\dots(\ad X)^{p_n-1}X$ instead of $(\ad X)^{p_n}(\ad Y)^{q_n-1}Y$; of course, if $q_n>1$, or if $q_n=0$ and $p_n>1$, then the term is zero.)\\
If $G$ is nilpotent of step $s$,
this sum is finite: only commutators of length $\leq s$ may be nonzero.
For example, for nilpotent groups of step $2$, the Campbell-Baker-Hausdorff formula reads
\[
X*Y=X+Y+\frac{1}{2}[X,Y]
\]
and for nilpotent groups of step 3
\[
X*Y=X+Y+\frac{1}{2}[X,Y]+\frac{1}{12}[X,[X,Y]]+\frac{1}{12}[Y,[Y,X]].
\]
From now on, whenever we are considering a connected simply connected
nilpotent Lie group $G$ we realize it as the algebra $\fg=\Lie(G)$
equipped with the multiplication $*$ given by the Campbell-Baker-Hausdorff formula.
The identity element of $G$ will therefore be denoted by~$0$.
Also, if $\fg$ is a nilpotent Lie algebra, we will denote by $(\fg,*)$ the corresponding connected simply connected nilpotent Lie group.

\subsection{Word maps on nilpotent Lie groups}
\label{subsec: word maps}

Let $F_{k}$ be the free non-abelian group on $k$ letters $x_{1},\dots,x_{k}$.
For any group $G$, any word $w\in F_{k}$ induces a map
\[
\begin{array}{lccc}
w_{G}:& G^{k} & \longrightarrow & G\\
& \bg=(g_{1},\dots,g_{n}) & \longmapsto & w(\bg)
\end{array}
\]
where $w(\bg)$ is the element of $G$ one gets by substituting
$g_{i}$ for $x_{i}$ in $w$. A \emph{word map $\omega$ on $k$ letters on $G$} is a map that can be obtained in this way: $\omega=w_G$ for some $w\in F_k$. For Lie groups, word maps are smooth, and we will see in this subsection that for nilpotent Lie groups, they are polynomial maps of degree bounded by
the nilpotency step of $G$.

\begin{definition}
Let $V$ and $W$ be two finite dimensional vector spaces. A map $f:V\rightarrow W$ is said to be \emph{polynomial of degree at most $d$} if it is a polynomial map of degree at most $d$ when expressed in some (or any) bases for $V$ and $W$.
\end{definition}

In the following lemma, $G=(\fg,*)$ is endowed with the vector space structure coming from the Lie algebra $\fg$.

\begin{lem}\label{polynomialwordmap}
Let $G$ be a connected simply connected Lie group, nilpotent of step $s$. For any word
$w\in F_{k}$, the map $w_{G}:G^k\rightarrow G$ is a polynomial map of degree at most $s$.
\end{lem}
\begin{proof}
The map $(X,Y)\mapsto [X,Y]$ is bilinear, and therefore any map of the form
$$(X_1,\dots,X_k)\mapsto [X_{i_1},[X_{i_2},\dots,[X_{i_{s-1}},X_{i_s}]\dots]]$$
is polynomial of degree at most $s$. Now, from the Campbell-Baker-Hausdorff formula, one sees that the word map $w_G$ can be expressed as a linear combination of such brackets, and this proves the lemma.
\end{proof}


\begin{definition}[Length of a word map]
For $\omega$ a word map on $k$ letters on $G$, set
$$l(\omega):=\min\{l(w)\,;\, w\in F_k \mbox{ and } w_G=\omega\}$$
and call it the \emph{length of $\omega$}. Here $l(w)$ denotes the length of the word $w$.
\end{definition}

\begin{lem}[Group of word maps]\label{mapgroup}
The set $F_{k,G}$ of word maps on $k$ letters on a group $G$ has a natural group structure, and we have the following properties:
\begin{enumerate}
\item The group $F_{k,G}$ is finitely generated.
\item If $G$ is a nilpotent group of step $s$, then so is $F_{k,G}$.
\item If $G$ is a connected Lie group, then, for almost every $k$-tuple $\bg=(g_1,g_2,\dots,g_k)\in G^k$, the subgroup $\Gamma_\bg$ generated by the subset $\{g_1,g_2,\dots,g_k\}$ is isomorphic to $F_{k,G}$.
\end{enumerate}
\end{lem}
\begin{proof}
Let $F_k$ be the abstract free group on $k$ letters $x_1,\ldots,x_k$. The group structure on $F_{k,G}$ is uniquely defined by the fact that the surjective map
$$\begin{array}{lccc}
\Phi: & F_k & \rightarrow & F_{k,G}\\
& w & \mapsto & w_G
\end{array}
$$
is a group homomorphism. Items $(1)$ and $(2)$ are clear. To justify $(3)$ note that if $w \notin \ker \Phi$, then the set of $\bg\in G^k$ such that $w(\bg)=0$ is a proper analytic subvariety of $G^k$, and hence has zero Lebesgue measure. They are only countably many such $w$'s, so we see that for almost every $k$-tuple $\bg\in G^k$, the unique group homomorphism $\theta_\bg:F_k\rightarrow \Gamma_\bg$ mapping $x_i$ to $g_i$ for $i\in\{1,\dots,k\}$ has kernel exactly $\ker \Phi$. Hence $\Gamma_\bg$ is isomorphic to $F_{k,G}$.
\end{proof}

Recall (cf. \cite{bassformula, guivarchformula, pansu}) that every finitely generated $s$-step nilpotent group $\Gamma=\langle S \rangle$ has polynomial growth, and more precisely there is an integer $\tau(\Gamma) \in \NN$ independent of the generating set $S$ such that $|B_\Gamma(n)| \sim c_S \cdot n^{\tau(\Gamma)}$ for some constant $c_S>0$. The growth exponent $\tau(\Gamma)$ is given by the Bass-Guivarc'h formula:

\begin{equation}\label{BGf}
\tau(\Gamma) = \sum_{k=1}^s k \cdot \rrk(\Gamma^{(k)}/\Gamma^{(k+1)}),
\end{equation}
where the $\Gamma^{(k)}$'s are the terms of the central descending series of $\Gamma$ and $\rrk$ the (torsion-free) rank of the corresponding abelian group. We conclude:

\begin{cor}[Generic growth for a random subgroup]\label{mapgrowth}
Let $\tau_k$ be the growth exponent of the finitely generated group $F_{k,G}$. Then, for almost every $\bg\in G^k$, the group $\Gamma_\bg$ has growth exponent~$\tau_k$.
\end{cor}

\subsection{Estimates on polynomial maps}

The purpose of this paragraph is to derive the elementary estimates on the measure of the set of points on which a polynomial takes small values.

Let $B$ be a convex subset of $\RR^n$, and $f:B\rightarrow \RR$  a real-valued polynomial function of degree at most $d$. For $\epsilon>0$, we set $Z_{\epsilon,B}(f):=\{x\in B; |f(x)|\leq\epsilon\}$. Then, from Brudnyi-Ganzburg \cite{brudnyiganzburg}, we have the following:

\begin{thm}[Remez-type inequality]\label{remez}
$$\sup_{x\in B}|f(x)|\leq \epsilon \cdot T_d\left(\frac{1+(1-|Z_{\epsilon,B}(f)|/|B|)^{\frac{1}{n}}}{1-(1-|Z_{\epsilon,B}(f)|/|B|)^{\frac{1}{n}}}\right)$$
where $T_d$ is the $d$-th Chebyshev polynomial of the first kind.
\end{thm}

Here $|B|$ denotes the Lebesgue measure of $B$ in $\RR^n$. It is well-known that those Remez-type inequalities imply the following estimate; we recall the proof for convenience of the reader.

\begin{prop}[Sublevel set estimate]
\label{prop:estimates on Polynomial maps}Fix $n_1,n_2,d \in \NN^*$. For any polynomial map $f: \RR^{n_1}\rightarrow \RR^{n_2}$  of degree at most $d$ and any convex subset of $B$ of $\RR^{n_1}$, one  has:
\begin{equation}\label{sublevelset}
|\{x \in B; \|f(x)\|\leq \epsilon\}|\leq 4\cdot  n_1 \left(\frac{\epsilon}{\|f\|_B}\right)^{\frac{1}{d}} |B|,
\end{equation}
where $\|f\|_B:=\sup_{x\in B}\|f(x)\|$, and $\|f(x)\|:=\max_{1\leq i \leq n_2} |f_i(x)|$ if $f(x)=(f_1(x),...,f_{n_2}(x)) \in \RR^{n_2}$.
\end{prop}
\begin{proof}  Clearly we may assume that $n_2=1$. Then the estimate follows immediately from the above Remez-type inequality after verifying the following two simple facts: first, for every $\eta \in [0,1]$, $\frac{1+(1-\eta)^{\frac{1}{n_1}}}{1-(1-\eta)^{\frac{1}{n_1}}}\leq \frac{2n_1}{\eta}$;  and second, $T_d(x)=\frac{1}{2}((x-\sqrt{x^2-1})^d + (x+\sqrt{x^2-1})^d )\leq \frac{1}{2}(2^d+1)x^d\leq 2^d x^d$ if $x \geq 1$. The first calculus fact can be seen as follows: using the change of variable $z=(1-\eta)^{\frac{1}{n_1}}$, the desired inequality is equivalent to $\frac{1+z}{1-z} \leq  2n_1/(1-z^{n_1})$, hence to $(1+z) \frac{1-z^{n_1}}{1-z} \leq 2n_1$. Since $z \in [0,1)$ and $\frac{1-z^{n_1}}{1-z}=1+z+\ldots+z^{n_1-1}$, the inequality follows.

\end{proof}

\begin{remark} The above bound goes back to Remez \cite{remez}. See \cite{Ga2000} for some historical discussion and related references.  We thank P. Varj\'u for providing these references. Note that the exponent $\frac{1}{d}$ is independent of the dimensions $n_1$ and $n_2$. This can be compared to \cite[\S 3]{kleinbock-margulis} and \cite[Lemma \S 3.4]{kleinbock-tomanov}, where a similar bound, albeit with a worse exponent $\frac{1}{dn_1}$ is derived. The fact that the exponent is independent of $n_1$ will be important in the proof of Theorem \ref{concludeth}.
\end{remark}

\section{A characterization of Diophantine nilpotent Lie groups}
\label{sec:characterization}

\subsection{The Borel-Cantelli Lemma and condition (PUB)}
\label{pubcondition}

Let $G=(\fg,*)$ be a connected simply connected nilpotent Lie group, endowed with a left-invariant Riemannian metric. Denote $n=\dim G$ and $s$ the nilpotency step of $G$. Throughout we identify $G$ with its Lie algebra $\fg$ via the exponential map. Then Lebesgue measure on $\fg$ coincides with Haar measure on $G$. In the following, we fix a basis of $\fg$ and the associated norm $\|x\|:=\max\{|x_i|\}$ on $\fg$.

\begin{definition}(Polynomial Uniform lower Bound)\label{pubdef}
We will say that the group $G$ satisfies \emph{condition (PUB) for words in $k$ letters} - or simply (PUB$_k$) - if there exist constants $C_k>0$ and $A_k>0$ such that
\begin{equation}
\forall \omega\in F_{k,G},\quad \left(\sup_{\mathbf{g}\in B_{G^k}(0,1)} \|\omega(\bg)\|\right) \geq C_k \cdot l(\omega)^{-A_k}. \tag{PUB$_k$}
\end{equation}
\end{definition}

This condition ensures that no word map contracts too much the unit ball of $G^k$. In other terms, it is a polynomial uniform lower bound on the size of word maps, whence the chosen abbreviation (PUB).

\begin{prop}\label{diophantinepub}
Let $G$ be a connected simply connected nilpotent Lie group. Then $G$ is Diophantine for words on $k$ letters if and only if it satisfies condition (PUB) for words on $k$ letters. In short,
\begin{center}
{\upshape
Diophantine$_k$ $\Longleftrightarrow$ (PUB$_k$).
}
\end{center}
\end{prop}
\begin{proof}
We start by the implication $(\Rightarrow)$ and prove the contrapositive. If $G$ does not satisfy (PUB$_k$), then there is a sequence of words $w_n$ in $F_{k}$ whose associated word maps on $G$ are nontrivial and satisfy $\|w_n(\bg)\|\leq \frac{1}{n}l(w_n)^{-n}$ for all $\bg \in B_{G^k}(0,1)$.
However the level sets $w_n^{-1}(\{0\})$ are proper subsets of zero Haar measure in $G^k$. The complement of the union of all these sets therefore has full measure in $G^k$. No $k$-tuple in this complement is Diophantine, hence $G$ is not Diophantine on $k$ letters.

\bigskip

Now we turn to the converse $(\Leftarrow)$. It is based on the Borel-Cantelli lemma.
For a word map $\omega$ on $G$, and $R\geq 1,\beta>0$, set
$$E_{\omega}(\beta):=\left\{\bg \in B_{G^{k}}(0,R) :\,\|\omega(\bg)\| \leq |B_{\Gamma_\bg}(l(\omega))|^{-\beta}\right\}.$$
Recall that $\tau_k:=\tau(\Gamma_\bg)$ is the growth exponent of $\Gamma_\bg$ for a random $\bg$ (i.e. $|B_{\Gamma_\bg}(n)|\sim c_k n^\tau_k,$ see Corollary \ref{mapgrowth}), that $s$ is the nilpotency  class of $G$ and that $A_k$ is the exponent appearing in the above definition of (PUB$_k$). We will show that if $\beta > s + \frac{A_k}{\tau_k}$ one has
$$\sum_{\omega\in F_{k,G}}|E_\omega(\beta)| < \infty.$$
In view of the  Borel-Cantelli Lemma, this is enough to conclude that $G$ is Diophantine for $k$ letters.
By Lemma~\ref{polynomialwordmap} we know that if $\omega$ is a word map on $G$, then it is polynomial of degree at most $s$. Therefore, if it is nontrivial, we may apply Proposition~\ref{sublevelset} with $f=\omega$, $B=B_{G^k}(0,R)$ and $\epsilon=\frac{c_k}{2}l(\omega)^{-\tau_k\beta}>0$ to get, assuming (PUB$_k$)
\begin{equation}\label{slvl}
|E_\omega(\beta)| \ll \left(\frac{l(\omega)^{-\tau_k\beta}}{\|\omega\|_B}\right)^{\frac{1}{s}} \ll l(\omega)^{\frac{-\tau_k\beta +A_k}{s}},
\end{equation}
where the implied constants depend only on $G$, $R$ and $k$. Therefore
$$\sum_{\omega\in F_{k,G}}|E_\omega(\beta)| \ll \sum_{\omega\in F_{k,G}} l(\omega)^{\frac{-\tau_k\beta +A_k}{s}}.$$
We now split this infinite sum in annuli $\{ \omega \in F_{k,G} ; 2^{m-1} \leq l(\omega) < 2^{m}\}$ noting that the size of each annulus is bounded above by the size of the ball of radius $2^{m}$ in $F_{k,G}$ and hence is a $O(2^{m\tau_k})$. It follows that the series converges for every $\beta>s+\frac{A_k}{\tau_k}$ as desired.
\end{proof}

\begin{remark}\label{bestexp} The proof shows that condition (PUB$_k$) with exponent $A_k$ implies that $G$ is Diophantine for words in $k$ letters, with exponent $s+A_k/\tau_k$. We will show later that the number $A_k$ can always be bounded above independently of $k$. As a consequence (see Proposition \ref{limsupbeta}), for every $\beta>s$, every Diophantine $s$-step nilpotent Lie group is $\beta$-Diophantine on $k$ letters if $k$ is large enough.
\end{remark}



\begin{remark}\label{rational}
If $G=\RR^d$, then word maps are of the form $(x_1,\dots,x_k)\mapsto \sum_{i=1}^kn_ix_i$, with $n_i\in\ZZ$, so condition (PUB) is clearly satisfied.
More generally, one can see, e.g. using the Campbell-Baker-Hausdorff formula, that word maps on rational nilpotent Lie groups have integer coefficients in an appropriate basis, so those groups certainly satisfy (PUB), and hence are Diophantine. This will also follow from the detailed analysis of condition (PUB) that we describe in \ref{wordmapslaws}.
\end{remark}

\subsection{An example: 2-step nilpotent groups}

In order to motivate the discussion of the next paragraph, we start by studying the case of 2-step nilpotent groups. We will show that all of them satisfy condition (PUB) and hence are Diophantine. Note that this is not a particular case of Remark~\ref{rational} above, as there exist non-rational nilpotent Lie groups of step two (see Raghunathan, \cite[Remark~2.14, page 38]{raghunathan}).

Let $G=(\fg,*)$ be a connected simply connected nilpotent Lie group of step 2, and let $k\in\bN$.
As $F_{k,G}$ is nilpotent of step 2, by \cite[\S 5.1]{MKS76}, any word map $\omega\in F_{k,G}$ has a representative of the form
\begin{equation}
w=\prod_{i=1}^{k}x_{i}^{e_{i}}\prod_{1\leq i<j\leq k}[x_{i},x_{j}]^{f_{i,j}}.\label{eq:word w}
\end{equation}
We will show that $\|\omega\|_B:=\sup_{\bg \in B_{G^k}(0,1)}{\|\omega(\bg)\|}$ admits a uniform lower bound when $\omega$ ranges over the set of word maps on $k$ letters on $G$. Clearly this will imply that $G$ satisfies (PUB$_k$) and hence is Diophantine thanks to Proposition \ref{diophantinepub}.

The map $\omega$ naturally induces a map on the abelianization, which under the identification $G\simeq \fg$ and in view of $(\ref{eq:word w})$ is just  $ \bar{\omega}: \fg^k  \to  \fg/\fg^{(2)}, (x_1,\dots,x_k)  \mapsto  \sum_{i=1}^k e_ix_i \mod \fg^{(2)}.$ If some $e_i$ is non zero, then obviously the map is not uniformly small on $B_{G^k}(0,1)$ and $\|\omega\|_B \gg 1$. If however all $e_i$'s are zero, then the map can be written $\omega: \fg^k  \to  \fg^{(2)}, (x_1,\dots,x_k)  \mapsto  \sum_{1\leq i<j\leq k}f_{ij}\cdot[x_{i},x_{j}]$, where $f_{ij} \in \ZZ$. Fixing $i<j$ and letting $x_m$ be the identity for all indices $m$ except $i$ and $j$, we see that $\|\omega\|_B \geq |f_{ij}| \sup_{x,y \in B_G(0,1)}\|[x,y]\|$. Since the $f_{ij}$'s are integers and not all zero, we obtain the desired uniform lower bound.

\bigskip

With analogous elementary methods, one can also treat the case of $3$-step nilpotent groups. The idea, as in the $2$-step case, consists in verifying that relations among word maps have to be rational. In particular all word maps sit inside a lattice in the space of polynomial maps $\fg^k\ra\fg$, and this yields the desired uniform lower bound on their norm. Unfortunately this approach fails in step $6$ and higher as we will demonstrate below.

We now turn to a more systematic study of the possible relations among word maps on a nilpotent Lie group, which will allow us, in section~\ref{sec:submodules}, not only to show that nilpotent groups of step $\leq 5$ satisfy (PUB), but also to construct non-Diophantine nilpotent Lie groups of any step $\geq 6$.

\subsection{Laws on nilpotent Lie algebras and the Diophantine property}\label{wordmapslaws}
Our goal in the remainder of this section will be to prove another characterization of the Diophantine property for a nilpotent Lie group $G$ in terms of the ideal of laws of the Lie algebra $\g$ (cf. Theorem \ref{diophantinelaws} below). Recall that $G$ is a connected simply connected nilpotent Lie group of step $s$. The group $F_{k,G}$ of word maps on $k$ letters on $G$ is by definition isomorphic to the quotient group $F_{k,s}/R_{k,s}(G)$, where $F_{k,s}$ is the free nilpotent group of step $s$ on $k$ generators, and $R_{k,s}(G)$ is the group of laws of $G$ in $F_{k,s}$, i.e. the normal subgroup of $F_{k,s}$ consisting of all (classes of) words that are identically zero on $G^k$.

\bigskip

For any ring $R$, we denote by $\cF_{k,s}(R)$ the free $s$-step nilpotent Lie algebra on the finite set $\{\sx_1,\dots,\sx_k\}$ over $R$. It is the quotient of the free Lie algebra by the ideal generated by commutators of length at least $s+1$. For $R=\RR$, we write $\cF_{k,s}=\cF_{k,s}(\RR)$.
The group $F_{k,s}$ sits naturally as a lattice in $(\cF_{k,s},*)$, the simply connected nilpotent Lie group associated to the free nilpotent Lie algebra of step $s$ over $k$ generators. 
The precise connection between group laws and Lie algebra laws is given by the following lemma. If $\sr \in \cF_{k,s}$ is written as a sum of homogeneous components $\sr=\sum_1^s \sr_i$, each $\sr_i$ being a linear combination of brackets of order $i$, then we set $|\sr|:=\max_1^s \|\sr_i\|^{\frac{1}{i}}$, where $\|\cdot\|$ is a norm on $\cF_{k,s}$.

\begin{lem}\label{lawtoword}
For each $s\in\bN^*$, there are positive integers $C,D$ such that
\begin{itemize}
\item If $w\in F_{k,s}$, then $w(e^{\sx_1},\dots,e^{\sx_k})=e^{\frac{1}{C}\sr(\sx_1,\dots,\sx_k)}$  for some $\sr\in\cF_{k,s}(\ZZ)$ with $|\sr|\leq D \cdot l(w)$.
\item If $\sr\in\cF_{k,s}(\ZZ)$, then $e^{C{\sr(\sx_1,\dots,\sx_k)}}=w(e^{\sx_1},\dots,e^{\sx_k})$ for some $w\in F_k$ with $l(w)\leq D \cdot|\sr|$.
\end{itemize}
\end{lem}
\begin{proof}
For the first item, one shows by induction on $l(w)$ that $w=e^{\sum_1^{s} \frac{1}{a_i}\sr_i}$, where $\sr_i \in \cF_{k,s}(\ZZ)$ is homogeneous of degree $i$ and $\|\sr_i\|^{\frac{1}{i}} \leq D \cdot l(w)$.  Here the $a_i$'s are positive integers which are fixed once and for all and independent of $w$. One picks the $a_i$ recursively in such a way that $a_1\ldots a_{i-1} b_i$ divides $a_i$ for each $i$, where $b_i$ is the lcm of the denominators of the coefficients appearing in front of the brackets of order at most $i$ in the Campbell-Baker-Hausdorff formula. The main point to observe is that due to this choice of $a_i$'s if $w=e^{\sum_1^{s} \frac{1}{a_i}\sr_i}$ for some $\sr_i$'s then $e^{\sx_j}w$ remains of this form when applying the Campbell-Baker-Hausdorff formula.

The second item is proved by induction on the nilpotency class $s$. Suppose it holds when the nilpotency class is $<s$, then write $\sr=\sr_{<s}+\sr_s$ the homogeneous components of order $<s$ and order $s$ respectively. From the induction hypothesis there is a word $w_{<s}$ of length $\leq D_{s-1}|\sr|$ such that for $C=C_{s-1}$, $e^{C\sr_{<s}}=w_{<s}$ modulo commutators of order $s$ in $F_{r,s}$. It then follows from the first item that $e^{C\sr_{<s}}=w_{<s}e^{\frac{1}{C}\sr'_s}$, where $\sr'_s$ is homogeneous of order $s$ and in $\cF_{k,s}(\ZZ)$. Moreover $|\sr'_s|\leq D_s l(w_{<s}) \leq D_s^2 |\sr|$.  Since $\sr_{<s}$ and $\sr'_s$ commute, we have $e^{C^2\sr}=w_{<s}^Ce^{\sr'_s+C^2\sr_s}$. It thus suffices to prove the assertion in the case when $\sr=\sr'_s$ is homogeneous of degree $s$. This follows from two observations. First if $c$ is a commutator of order $s$ in $\cF_{k,s}$ with letters $\sx_j$, then $e^c$ coincides with the group commutator of the same form in the letters $e^{\sx_j}$'s. Denote by $n$ the least integer greater than $|\sr|$. The second remark is that every positive integer $m < n^s$ can be written in base $n$ as a sum of at most $s$ terms of the form $an^{i}$ for some integers $i=0,\ldots,s-1$ and $a \in [0,n-1]$. This shows that for every commutator shape $c(\sx_1,\ldots,\sx_k)$ one can write $e^{mc}$ as a product of at most $s$ group commutators in the variables $(e^{\sx_j})^\ell$, where $|\ell|\leq n$. The result follows immediately.

\end{proof}

For future reference, we record the following observation, which was used in the proof: if $\sr$ in the second item of Lemma \ref{lawtoword} is a bracket commutator of order $s$ and shape $c(\sx_1,\ldots,\sx_k)$ (so that it belongs to $\cF_{k,s}^{[s]}$, the ideal of homogeneous elements of degree $s$ in $\cF_{k,s}$), then the associated word $w \in F_k$ given by the Lemma can be taken to be the commutator of order $s$ with exactly the same shape (e.g. if $\sr(\sx_1,\sx_2)=[\sx_1,[\sx_1,\sx_2]]$, then $w(g_1,g_2):=(g_1,(g_1,g_2))$ works, where $(a,b)=aba^{-1}b^{-1}$ is the group commutator). This is a simple consequence of the Campbell-Baker-Hausdorff forumla (see after Prop. \ref{prop:exp Sur} above). Conversely, given a commutator word $w$ of order $s$, we have $w(e^{\sx_1},\ldots,e^{\sx_k})=e^{\sr(\sx_1,\ldots,\sx_k)}$ for all $\sx_i$'s in $\cF_{k,s}$, where $\sr$ is the bracket commutator of the same shape.

In particular, if $\sr$ belongs to $\cF_{k,s}^{(i)}$, the $i$-th term in the central descending series of $\cF_{k,s}$, then $w$ can be chosen to belong to $F_k^{(i)}$, the $i$-th term in the central descending series of $F_k$. Similarly, if $\sr$ belongs to $D^{(i)}(\cF_{k,s})$, the $i$-th term in the derived series of $\cF_{k,s}$, then $w$ can just as well be taken to belong to the $i$-th term in the derived series of $F_k$. And conversely the same holds in the opposite direction, i.e. given $w$ in $F_k^{(i)}$ (resp. $D^{(i)}(F_k)$), we may find $\sr$  in $\cF_{k,s}^{(i)}$ (resp. $D^{(i)}(\cF_{k,s})$) satisfying the conclusion of the first item in Lemma \ref{lawtoword}.

Lemma \ref{lawtoword} is fairly standard. We added a proof for the reader's convenience. More information on this topic and on the geometry of nilpotent groups, is contained in Tits' appendix to \cite{gromov} and in the second author's paper \cite[\S 2]{breuillard}.

\bigskip

If $\fg$ is an arbitrary real nilpotent Lie algebra with nilpotency class $\leq s$, for each $\bX=(X_1,\dots,X_k)\in \fg^k$, there is a unique Lie algebra homomorphism $\theta_\bX:\cF_{k,s}\rightarrow \fg$ such that for each $i$, $\theta_\bX(\sx_i)=X_i$.

\begin{definition}
The \emph{ideal of laws on $k$ letters on $\fg$} is defined to be
$$\cL_{k,s}(\fg):=\bigcap_{\bX\in\fg^k} \ker\theta_\bX.$$
\end{definition}

The quotient $\cF_{k,s}/\cL_{k,s}(\fg)$ is the relatively free Lie algebra in the variety of $k$-generated Lie algebras associated to $\fg$.

\begin{prop}
The set $\cL_{k,s}(\fg)$ is a fully invariant ideal of $\cF_{k,s}$, i.e. $\cL_{k,s}(\fg)$ is an ideal of $\cF_{k,s}$, and it is stable under all Lie algebra endomorphisms of $\cF_{k,s}$. Conversely, if $\cL$ is a fully invariant ideal of $\cF_{k,s}$ there exists a nilpotent Lie algebra $\fg$ with nilpotency class $\leq s$ such that $\cL_{k,s}(\fg)=\cL$.
\end{prop}
\begin{proof}
By definition, $\cL_{k,s}(\fg)$ is the intersection of all ideals $\ker\theta_\bX$, so it is itself an ideal. Let $\varphi$ be a Lie algebra endomorphism of $\cF_{k,s}$. For each $i$ in $\{1,\dots,k\}$, denote $\sr_i=\sr_i(\sx_1,\dots,\sx_k)$ the image of $\sx_i$ under $\varphi$. For $\bX\in\fg^k$, with a slight abuse of notation, we set $\varphi(\bX):=(\sr_1(\bX),\dots,\sr_k(\bX))$. The homomorphism of Lie algebras $\theta_\bX\circ\varphi:\cF_{k,s}\rightarrow\fg$ maps each $\sx_i$ to $\sr_i(\bX)$ and therefore,
$$\theta_\bX\circ\varphi = \theta_{\varphi(\bX)}.$$
This shows that $\cL_{k,s}(\fg)$ is stable under $\varphi$.\\
Conversely, if $\cL$ is any ideal of the free Lie algebra $\cF_{k,s}$, the ideal of laws in $k$ letters of the quotient Lie algebra $\cF_{k,s}/\cL$ is the smallest fully invariant ideal of $\cF_{k,s}$ containing $\cL$. In particular if $\cL$ is a fully invariant ideal of $\cF_{k,s}$, then the ideal of laws of the quotient Lie algebra $\cF_{k,s}/\cL$ is exactly $\cL$.
\end{proof}

\begin{definition}
The \emph{ideal of rational laws on $k$ letters on $\fg$} is defined to be the real vector space $\cL_{k,s}(\fg)_\QQ$ generated by $\cL_{k,s}(\fg)\cap\cF_{k,s}(\QQ)$. It is also a fully-invariant ideal of $\cF_{k,s}$.
\end{definition}

Denoting  by $(\fg,*)$ the simply connected Lie group associated to the Lie algebra $\fg$, we have the following.

\begin{prop}\label{groupofwordmaps}
The group of word maps $F_{k,G}$ naturally embeds into two Lie groups:
\begin{itemize}
\item It is a finitely generated subgroup of $(\cF_{k,s}/\cL_{k,s}(\fg),*)$.
\item It is a lattice in $(\cF_{k,s}/\cL_{k,s}(\fg)_\QQ,*)$.
\end{itemize}
\end{prop}

We defer the proof to later in this section. The discrepancy between $\cF_{k,s}/\cL_{k,s}(\fg)_\QQ$ and $\cF_{k,s}/\cL_{k,s}(\fg)$ lies at the heart of the property of being Diophantine for $G$. We have a natural epimorphism of real Lie algebras:
\begin{equation}\label{thelambdamap}
\Lambda: \cF_{k,s}/\cL_{k,s}(\fg)_\QQ \to \cF_{k,s}/\cL_{k,s}(\fg).
\end{equation}
Before we state the main result of this section we require one more definition.

\begin{definition}(Diophantine subspace)\label{diophantine-subspace} Let $V$ be a finite-dimensional $\QQ$-vector space and choose a norm $\|\cdot\|$ on $V(\RR)$. A real subspace $L$ of $V(\RR)$ is said to be Diophantine in $V$ if there are constants $C,A >0$ such that $$d(v,L) \geq C \cdot \|v\|^{-A}$$
for every $v \in V(\ZZ)\setminus L$.
\end{definition}

\begin{example}(Subspaces defined over a number field)\label{algebraic} Let $K \leq \RR$ be a finite field extension of $\QQ$. If $L$ has a basis made of vectors in $V(K)$, then $L$ is a Diophantine subspace of $V$. Indeed $L$ is also the intersection of the kernels of linear forms $\ell_1,\ldots,\ell_{\codim L}$ on $V$ which are all defined over $K$ and we are left to verify that $|\ell_i(v)| \gg \|v\|^{-A}$ for each $v \in V(\ZZ)\setminus\{\ker \ell_i\}$. This is of course well-known, and can easily be verified from the product formula $h(x)=h(x^{-1})$ and the height bounds $h(xy)\leq h(x)+h(y)$ and $h(x+y)\leq h(x)+h(y)+\log 2$, where $h(x)$ is the absolute Weil height for the algebraic number $x$ (cf. \cite{bombieri-gubler}).
\end{example}

We now have:

\begin{thm}\label{diophantinelaws}
Let $G$ be a connected simply connected nilpotent Lie group with Lie algebra $\fg$ and $k \geq 1$. The following are equivalent:
\begin{enumerate}
\item \label{group} The group $G$ is Diophantine for words in $k$ letters.
\item \label{wordmaps} The group $F_{k,G}$ is Diophantine as a subgroup of $(\cF_{k,s}/\cL_{k,s}(\fg),*)$.
\item \label{ideal} The ideal of laws $\cL_{k,s}(\fg)$ is Diophantine in $\cF_{k,s}$.
\item \label{kernel} $\ker \Lambda$ is Diophantine in $\cF_{k,s}/\cL_{k,s}(\fg)_\QQ$.
\end{enumerate}
\end{thm}

\bigskip

Note that the condition that $F_{k,G}$ be a Diophantine subgroup of $(\cF_{k,s}/\cL_{k,s}(\fg),*)$ (i.e. condition $(2)$ in Theorem \ref{diophantinelaws}) is just a reformulation of (PUB$_k$). Indeed by Lemma \ref{lawtoword} for each word map $\omega \in F_{k,G}$, there is an element $\sr_\omega \in \cF_{k,s}/\cL_{k,s}(\fg)$ such that $\omega(\bg)=e^{\sr_\omega(\log \bg)}$ with $l(\omega) \lesssim \|\sr_\omega\| \lesssim l(\omega)^s$ and $\sup_{B_{G^k}(0,1)} \|\omega(\bg)\| = \sup_{\bx \in B_{\fg^k}(0,1)}\|\sr_\omega(\bx)\|$ is a norm on the Lie algebra $\cF_{k,s}/\cL_{k,s}(\fg)$. All norms are equivalent, so (PUB$_k$) (as formulated in Definition \ref{pubdef}) precisely means that $F_{k,G}$ is Diophantine.\\

We are now ready to prove Proposition \ref{groupofwordmaps} and Theorem \ref{diophantinelaws}.

\begin{proof}[Proof of Prop. \ref{groupofwordmaps}]
As follows from Lemma \ref{lawtoword}, $F_{k,s}$ is a lattice in $(\cF_{k,s},*)$. Moreover the kernel of the natural group homomorphism $F_{k,s} \to (\cF_{k,s}/\cL_{k,s}(\fg),*)$ coincides with $R_{k,s}(G)$. Hence $F_{k,G}=F_{k,s}/R_{k,s}(G)$ is naturally a subgroup of $(\cF_{k,s}/\cL_{k,s}(\fg),*)$.

Similarly Lemma~\ref{lawtoword} implies that $R_{k,s}(G)$ is a lattice in $(\cL_{k,s}(\fg)_\QQ,*)$. This implies (see \cite[Lemma~1.16, page~25]{raghunathan}) that $F_{k,s}/R_{k,s}(G)$ is a lattice in $(\cF_{k,s}/\cL_{k,s}(\fg)_\QQ,*)$.
\end{proof}

\begin{proof}[Proof of Theorem \ref{diophantinelaws}] We just observed that $(2)$ is a reformulation of (PUB$_k$). Hence Proposition \ref{diophantinepub} shows that $(1)$ and $(2)$ are equivalent.

If $W \leq V$ are finite-dimensional $\QQ$-vector spaces and $L$ is a real subspace $W(\RR) \leq L \leq V(\RR)$, then it follows easily from Definition \ref{diophantine-subspace} that $L$ is Diophantine in $V$ if and only if $L/W$ is Diophantine in $V/W$. This yields the equivalence $(3) \Leftrightarrow (4)$.

We now prove the equivalence $(2) \Leftrightarrow (4)$. Let $N:=(\cF_{k,s}/\cL_{k,s}(\fg)_\QQ,*)$ and the discrete lattice $\Gamma:=F_{k,s}/R_{k,s}(G)$ inside it. Let $L:=\ker \Lambda$. We need to prove that $\Gamma$ is Diophantine in $N/L$ if and only if the Lie algebra $\Lie(L)$ is a  Diophantine subspace $\Lie(N)$ in the sense of Definition \ref{diophantine-subspace}. This again follows easily from Lemma \ref{lawtoword} and the Campbell-Baker-Hausdorff (CBH) formula. Indeed, if $\Gamma$ is not Diophantine in $N$, then for every $A>0$ there are words $w$ such that $w(\sx_1,\ldots,\sx_k)=e^{\su_w}e^\ell$, where $\su_w \in \Lie(N)\setminus\{0\}$ is very small, i.e. $\|\su_w\|\ll l(w)^{-A}$, and $\ell \in \Lie(L)$.  By Lemma \ref{lawtoword}, this implies that there are integer points $\sr \in \Lie(N)(\ZZ)$ with $\|\sr\| \ll l(w)^s$ such that $e^{\sr}=e^{C\su_w}e^{C\ell}$  for a constant $C$. Applying the CBH formula, using the fact that $\Lie(L)$ is an ideal, we see that $\|\sr-C\ell\|\ll\|\sr\|^{-A/s}$, showing that $\Lie(L)$ is not Diophantine as a subspace of  $\Lie(N)$. The reverse direction is similar and we omit it.
\end{proof}

\begin{remark}\label{add}
Note that the proof shows that if $\cL_{k,s}(\fg)$ is Diophantine in $\cF_{k,s}$ with exponent $A$, then $G$ satisfies (PUB$_k$) with exponent $A_k=sA$, and therefore, by Remark~\ref{bestexp}, $G$ is Diophantine in $k$ letters with exponent $s+\frac{sA}{\tau_k}$.
\end{remark}

\subsection{Nilpotent Lie groups defined over an algebraic number field}

A real Lie algebra $\fg$ is said to be defined over a proper subfield $K \leq \RR$, if one can find a basis $\{X_i\}_{i=1}^d$ such that the associated structure constants of $\fg$ (that is the numbers $c_{ijk}$ such that $[X_i,X_j]=\sum_{k=1}^d c_{ijk}X_k$) belong to $K$. We will also say that the associated connected simply connected nilpotent real Lie group is defined over $K$ when its Lie algebra is so.


One readily checks that if the $s$-step nilpotent Lie algebra $\fg$ is defined over $K$, then its ideal of laws $\cL_{k,s}(\fg)$ is also defined over $K$. The following is then a direct consequence of Theorem~\ref{diophantinelaws}.

\begin{cor}
If $G$ is a connected simply connected nilpotent Lie group defined over a number field $K$ ($[K:\QQ]<\infty$), then $G$ is Diophantine.
\end{cor}
\begin{proof}
The ideal of laws $\cL_k(\fg)$ is defined over the number field $K$, and therefore, by Example~\ref{algebraic}, it must be Diophantine. Theorem~\ref{diophantinelaws} then implies that the group $G$ is Diophantine.
\end{proof}

\subsection{Connected, but non-simply connected nilpotent Lie groups}

We explain here how the above can be adapted to handle non simply connected connected nilpotent Lie groups $G$ as well.\\

Let $\tilde{G}$ be the universal cover of the connected nilpotent Lie group $G$, so that $G=\tilde{G}/Z$,  where $Z$ is a discrete subgroup of $\tilde{G}$ contained in its center. The group $Z$ is a torsion-free abelian group, say of rank $r$.

A first observation is that the groups of words maps $F_{k,G}$ and $F_{k,\tilde{G}}$ are naturally isomorphic: indeed every law on $G$ must also be a law on $\tilde{G}$ because $Z$ is discrete in $\tilde{G}$. Second we prove the following:

\begin{thm}
Let $G$ be a connected nilpotent Lie group and $\tilde{G}$ its universal cover. Then $G$ is Diophantine on $k$ letters if and only if $\tilde{G}$ is Diophantine on $k$ letters.
\end{thm}

\begin{proof} One direction is obvious: if $G$ is Diophantine, then so is $\tilde{G}$. In the converse direction, we use the characterization in terms of the property (PUB$_k$) of Proposition \ref{diophantinepub} and modify the Borel-Cantelli argument used in the proof of this proposition. What needs to be estimated is the Haar measure of the sets $E'_\omega(\beta):=\{\bg \in B_{\tilde{G}^k}(0,R); d(\omega(\bg),Z) < |B_{F_{k,G}}(n)|^{-\beta}\}$. This splits into a union of at most $O(l(\omega)^{rs})$ subsets of the form $E'_\omega(\beta,z):=\{\bg \in B_{\tilde{G}^k}(0,R); \|z^{-1}\omega(\bg)\| < |B_{F_{k,G}}(n)|^{-\beta}\}$ for $z \in Z\setminus\{0\}$. Since $Z$ is discrete, the quantity $\sup_{\bg \in B_{\tilde{G}^k}(0,R)}{\|z^{-1}\omega(\bg)\|}$ is bounded away from $0$ uniformly in $z \neq 0$. Applying the Remez-type inequality (Prop. \ref{prop:estimates on Polynomial maps}) to each polynomial map $\bg \mapsto z^{-1}\omega(\bg)$ and using condition (PUB$_k$) for $\tilde{G}$, we obtain
 $$|E'_\omega(\beta)| \leq |E'_\omega(\beta,0)| + \sum_{z \neq 0} |E'_\omega(\beta,z)| \ll l(\omega)^{-\frac{\beta \tau_k - A_k}{s}} + l(\omega)^{rs} l(\omega)^{-\frac{\beta \tau_k}{s}}$$
 and the series converges as soon as $\beta> s + \frac{\max\{rs, A_k\}}{\tau_k}$. We conclude that $G$ is Diophantine.
\end{proof}







\section{Fully invariant ideals of the free Lie algebra $\mathcal F_k$}
\label{sec:submodules}

In this section, we complete our study of the Diophantine property for nilpotent Lie groups and explain the connection between the Diophantine property for $s$-step nilpotent Lie algebras and the absence of multiplicity in the ideal of laws $\cL_{k,s}$ viewed as a module over $SL_k$. We then complete the proof of the results stated in the introduction.


\subsection{$\mathcal{F}_{k,s}$ as an $SL_k$-module}
Recall that $\cF_k$ denotes the free Lie algebra on $k$ generators and $\cF_{k,s}=\cF_k/\cF_k^{(s+1)}$ the $s$-step free nilpotent Lie algebra. Here $\cF_k^{(i)}$ denotes the $i$-th term of the central descending series of $\cF_k$. The ring $\End \cF_{k,s}$ of Lie algebra endomorphisms of $\cF_{k,s}$ acts naturally on $\cF_{k,s}$, so that $\cF_{k,s}$ has a structure of an $\End \cF_{k,s}$-module. An $\End \cF_k$-submodule of $\cF_{k,s}$ is called a \emph{fully invariant ideal} of $\cF_{k,s}$.

Below we will show that for $k\geq 3$ (resp. $k=2$) for all $s\geq 6$ (resp. $s\geq 7$), there exists a fully invariant ideal of $\cF_{k,s}(\RR)$ which is not Diophantine. By Theorem \ref{diophantinelaws} this will show existence of non-Diophantine nilpotent Lie groups.

The group $SL_k$ acts on $\cF_{k,s}$ by linear substitution of the free variables, and thus embeds naturally in $\End \cF_{k,s}$.

For $s\geq 1$, we denote $\cF_{k}^{[s]}$ the subspace of $\cF_{k,s}$ consisting of homogeneous elements of degree $s$.
Note that $\cF_{k}^{[s]}$ is stable under the action of $SL_k$ and that a vector subspace $V \leq  \cF_{k}^{[s]}$ is invariant under the action of $SL_k$ if and only if it is a fully invariant ideal of $\cF_{k,s}$. So in order to build fully invariant ideals in $\cF_{k,s}$ we can look for $SL_k$-invariant subspaces of $\cF_{k}^{[s]}$.
Our first observation is the following.

\begin{lem}(complete reducibility)\label{qsubmodules}
The $SL_k$-module $\cF_k^{[s]}$ is completely reducible, i.e. there are positive integers $n_i$ such that
$$\cF_k^{[s]}=\bigoplus_{i=0}^l V_i^{n_i},$$
where each $V_i$ is an irreducible highest-weight $SL_k$-module defined over $\QQ$ and $V_i\not\simeq V_j$ if $i\neq j$.
\end{lem}
\begin{proof}
This follows Weyl's complete reducibility theorem. See Serre~\cite[Part I, Chapter~6, \S 3]{serrelalg}.
\end{proof}


\subsection{Multiplicity and Diophantine submodules}

We now want to find under which condition $\cF_k^{[s]}(\RR)$, the subspace of $\cF_{k,s}(\RR)$ consisting of homogeneous elements of degree $s$, admits a non-Diophantine $SL_k$-submodule. Say that $\cF_k^{[s]}(\RR)$ is \emph{multiplicity free} if in the decomposition given by the above lemma, $n_i=1$ for all $i$. If not, we say that $\cF_k^{[s]}(\RR)$ admits multiplicity. The following simple observation is key to our proofs:

\begin{lem}\label{lem:mult free is dioph}
Let $k,s\in\mathbb N$.
\begin{enumerate}
\item If $\cF_k^{[s]}(\RR)$ is multiplicity-free, then it has only finitely many $SL_k$-submodules, all of which are defined over $\QQ$, and hence Diophantine.
\item If $\cF_k^{[s]}(\RR)$ admits multiplicity, then it has a non-Diophantine $SL_k$-submodule.
\end{enumerate}
\end{lem}
\begin{proof}
If $\cF_k^{[s]}(\RR)$ is multiplicity free, then, using notation of Lemma~\ref{qsubmodules}, we see that every $SL_k$-submodule $V$ has the form $V=\bigoplus_{i\in I} V_i$ where $I\subset\{1,\dots,n\}$. This certainly implies that they all are defined over $\QQ$. Example~\ref{algebraic} then shows that they are Diophantine.\\
Conversely, suppose $\cF_k^{[s]}$ admits multiplicity. Without loss of generality, we may assume $n_1\geq 2$ so that $\cF_k^{[s]}$ admits a submodule of the form $V_1\oplus V_1'$, with $V_1\simeq V_1'$, both of them being defined over $\QQ$ as subspaces of $\cF_k$. Fix an isomorphism $\alpha:V_1\rightarrow V_1'$ mapping $V_1\cap\cF_k(\ZZ)$ to $V_1'\cap\cF_k(\ZZ)$. Then choose some Liouville number $\lambda \in \RR$ --- i.e. some number such that the inequalites
\begin{equation}\label{liouville}
0<|\lambda-\frac{p}{q}|\leq q^{-q-1}
\end{equation}
have infinitely many integer solutions in $(p,q)$
--- and define
$$L_\lambda=\left\{x+\lambda\alpha(x):x\in V_1\right\}\subset \cF_k^{[s]}(\RR).$$
This is an $SL_k$-submodule of $\cF_k^{[s]}(\RR)$, which we claim to be non-Diophantine. To see this, take a non zero vector $x\in V_1\cap\cF_k(\ZZ)$, and let, for $p,q\in\ZZ$, $\sr_{p,q}:=qx+p\alpha(x)\in\cF_k^{[s]}(\ZZ)$. Then, for $p,q$ large enough in the set of solutions to (\ref{liouville}), we have
\begin{align}
0 < d(\sr_{p,q},L_\lambda) & \leq \|qx_1+p\alpha(x_1)-q(x_1+\lambda \alpha(x_1))\| \leq |p-q\lambda|\|\alpha(x_1)\| \leq q^{-q}.
\end{align}
As $\|\sr_{p,q}\|\ll q$, this proves what we wanted.
\end{proof}


\subsection{Applications}

As we explain in the appendix,  using Witt's Character Formula for the free Lie algebra, one may determine precisely when the $SL_k$-module $\cF_k^{[s]}$ is multiplicity-free. The conclusion is the following (Theorem~\ref{multfree}):

\begin{thm}
The $SL_k$-module $\cF_k^{[s]}$ is multiplicity-free if and only if $s\leq 5$ when $k\geq 3$ and if and only if $s\leq 6$ when $k=2$.
\end{thm}

This will allow us to derive Theorems~\ref{nondiophantinei} and \ref{metabeliani} announced in the introduction.

\begin{thm}\label{nondiophantine}
Fix an integer $k\geq 3$ (resp. $k=2$).
\begin{enumerate}
\item Every connected nilpotent Lie group of step $s\leq 5$ (resp. $s\leq 6$) is Diophantine on $k$ letters.
\item For every $s\geq 6$ (resp. $s\geq 7$), there are $s$-step nilpotent Lie groups which are not Diophantine on $k$ letters.
\end{enumerate}
\end{thm}
\begin{proof}
We only deal with the case $k\geq 3$, because the case $k=2$ is entirely analogous. Let $G$ be a connected nilpotent Lie group of step $s\leq 5$. From Theorem~\ref{diophantinelaws}, it suffices to show that $\cL_{k,s}=\cL_{k,s}(\fg)$ is Diophantine in $\cF_{k,s}$. Now $\cL_{k,s}$ is a fully invariant ideal of $\cF_{k,s}$ and therefore can be decomposed as
$$\cL_{k,s}=\bigoplus_{r\geq 1}\cL_{k,s}^{[r]},$$
where $\cL_{k,s}^{[r]}$ is the set of elements of $\cL_{k,s}$ of homogeneous degree $r$. For each $r$, $\cL_{k,s}^{[r]}$ is an $SL_k$-submodule of $\cF_{k,s}^{[r]}$. Combining Lemma~\ref{lem:mult free is dioph} and Theorem~\ref{multfree}, we get that for each $r\leq 5$, $\cL_{k,s}^{[r]}$ is defined over $\QQ$. Thus, $\cL_{k,s}$ is defined over $\QQ$ and hence Diophantine. This proves the first part of the theorem in the case $k\geq 3$.

Now let $k\geq 3$ and $s\geq 6$. From Theorem~\ref{multfree} and Lemma~\ref{lem:mult free is dioph}, we may find in $\cF_k^{[s]}(\RR)$ an $SL_k$-submodule $L$ that is non-Diophantine as a subspace of $\cF_{k,s}$. This is a fully invariant ideal of $\cF_{k,s}$ which is not Diophantine. Let $G$ be the connected simply connected Lie group with Lie algebra $\cF_{k,s}/\cL_{k,s}$. Then $G$ is $s$-step nilpotent, and its ideal of laws on $k$ letters is $\cL_{k,s}$ so that by Theorem~\ref{diophantinelaws}, $G$ is not Diophantine.
\end{proof}


The proof of Theorem~\ref{metabeliani} follows similar lines, the only new input is the fact, proved in Lemma \ref{irreducible} below, that the free metabelian Lie algebra is multiplicity-free as an $SL_k$ module.

\begin{thm}\label{metabelian}
Every connected nilpotent metabelian Lie group is Diophantine.
\end{thm}
\begin{proof}
Let $G$ be such a Lie group with nilpotency step $s$ and let $\fg$ be its Lie algebra. Let $\cL_k=\cL_k(\fg)$ the ideal of laws on $k$ letters in $\fg$. Since $G$ is metabelian, for each $r$, $\cL_k^{[r]}$ contains $\cM_k^{[r]}$ (see the notation of Lemma \ref{irreducible}). It then follows from this lemma that $\cL_k^{[r]}$ is equal either to $\cM_k^{[r]}$ or to $\cF_k^{[r]}$. In particular it is defined over $\QQ$. Hence so is $\cL_{k,s}(\fg)$ in $\cF_{k,s}$. By Example \ref{algebraic}, it must then be Diophantine in $\cF_{k,s}$ and Theorem~\ref{diophantinelaws} implies that $G$ is Diophantine.
\end{proof}

\begin{remark} Observe that the non-Diophantine nilpotent Lie groups constructed in Theorem \ref{nondiophantine} in step $s=6$ (or $s=7$ if $k=2$) are solvable of derived length $3$ (indeed $D^3(\cF_k)\subset \cF_k^{(8)}$).
\end{remark}

We can now build a non-Diophantine solvable but not nilpotent group.

\begin{thm}\label{thm:nonDiophgrps}
There exists a non-Diophantine solvable Lie group which is not nilpotent.
\end{thm}
\begin{proof}
Fix $s\geq 7$. Denote $\cM_2=D^2(\cF_k)$ the second term of the derived series of $\cF_k$. From Theorem~\ref{multfree} in the appendix, we know that $\cF_2^{[s]}$ admits multiplicity. By Lemma~\ref{irreducible}, this implies that in fact $\cM_2^{[s]}$ has multiplicity.
From there, using a Liouville number as in the proof of Lemma~\ref{lem:mult free is dioph}, we build an $SL_2$-submodule $\cL$ of $\cM_2^{[s]}$ and a sequence $(\sr_n)_{n\geq 1}$ of elements of $\cM_2^{[s]}(\ZZ)$ such that $\|\sr_n\|\rightarrow\infty$ and for each $n$,
\begin{equation}\label{closerelation}
0<d(\sr_n,\cL)<\|\sr_n\|^{-\|\sr_n\|^s}.
\end{equation}
By Lemma~\ref{lawtoword}, we may obtain from $(\sr_n)$ a sequence of words in two letters
$$w_n=e^{C\sr_n}\mod e^{\cF_2^{(s+1)}}$$
with $l(w_n)\leq C\cdot\|\sr_n\|^s$. Moreover, given that the $\sr_n$'s are in $\cM_2$, it follows from the remark made right after the proof of Lemma~\ref{lawtoword} that the $w_n$'s can be chosen in $M_2=D^2(F_2)$, the second term of the derived series of the free group $F_2$. This implies in particular that for any metabelian group $M$, all word maps $w_{n,M}$ are trivial.
Now take $M$ any metabelian non-nilpotent Lie group --- e.g. the group of affine transformations of the real line --- and let $N$ be the connected simply connected nilpotent Lie group with Lie algebra $\cF_2/(\cL\oplus\cF_2^{(s+1)})$. Let $G:=M\times N$. The word maps $w_{n,G}$ are trivial on $M\times\{1\}$, and therefore, the bound (\ref{closerelation}) shows that for any $\beta>0$,
$$0<\sup_{\bg\in B_{G^2}(0,1)} d(w_n(\bg),0)\ll 4^{-\beta l(w_n)}\ll |B_{\Gamma_{\bg}}(l(w_n))|^{-\beta}.$$
This shows that $G$ cannot be Diophantine, and of course $G$ is solvable non-nilpotent.
\end{proof}

\section{Concluding remarks}

\subsection{Baire category genericity}
We prove here Theorem \ref{nondiophtuples}. It is well-known that although almost every real number is Diophantine, there is a $G_\delta$-dense set of real numbers which are not. So topological and measure-theoretic genericity are very different notions. For $k$-tuples on nilpotent groups a similar phenomenon takes place:

\begin{prop}\label{baire} Let $G$ be a connected nilpotent real Lie group. If $k> \dim G/[G,G]$, then there is a $G_\delta$-dense set $D$ in $G^k$ of $k$-tuples which are not Diophantine.
\end{prop}
For the analogous result on $SU(2)$ see \cite{gamburd-hoory}. In fact $D$ can be chosen so that the $k$-tuples in $D$ are as non-Diophantine as possible, namely the speed of an approximation to $1$ by a word of length $n$ can be arbitrarily fast in $n$. The proof is based on the following lemma. Recall that $F_{k,G}$ is the relatively free group on $k$ generators associated to $G$ (see \S \ref{subsec: word maps}).

\begin{lem}\label{generic}If there is a dense subset $D_0$ of $G^k$ such that for each $\bg \in D_0$ the induced natural map $F_{k,G} \to G$ is not injective, then there is a $G_\delta$-dense set $D$ of non-Diophantine $k$-tuples.
\end{lem}

\begin{proof} Let $\omega_{\bg} \in F_{k,G} \setminus\{1\}$ be such that $\omega_{\bg}(\bg)=1$. Then $\omega_\bg^{-1}(1)$ is a proper analytic subvariety of $G^k$. In particular for every integer $n \geq 1$, the subset $O_{n,\bg}:=B(\bg,e^{-nl(\omega_{\bg})}) \setminus w_\bg^{-1}(1)$ of the open ball $B(\bg,e^{-nl(\omega_{\bg})})$ is open and its closure contains $\bg$. Set $D:=\cap_{n \geq 1} \cup_{\bg \in D_0} O_{n,\bg}.$
\end{proof}

\begin{proof}[Proof of Prop. \ref{baire}]Note that we may assume that $G$ is simply connected. Let $c \in \cF_{k,s}^{[s]}$ be a commutator of length $s$ not belonging to the ideal of laws $\cL_{k,s}(\fg)$. Let $D_0$ be the set of $k$-tuples $(x_1,\ldots,x_k)$ in $G^k$ such that $(x_1,\ldots,x_{k-1})$ span $\fg$ modulo $\fg^{(2)}$ (via the identification $G\sim \fg$), and such that $x_k$ lies in the $\QQ$-span of $(x_1,\ldots,x_{k-1})$ modulo $\fg^{(2)}$. Since $k>\dim \fg/\fg^{(2)}$ this set is clearly dense in $G^k$. By definition of $D_0$, and using the multi-linearity of $c$, if $\bg \in D_0$, then there are integers $n_i \in \ZZ$ with $n_k \neq 0$ such that $n_kc(x_1,\ldots,x_{k})=\sum_{i=1}^{k-1} n_i c(x_1,\ldots,x_{k-1},x_i)$. However viewed as an element of $\cF_{k,s}$, the quantity
$$r(\sx_1,\ldots,\sx_{k}):=n_kc(\sx_1,\ldots,\sx_{k})-\sum_{i=1}^{k-1} n_i c(\sx_1,\ldots,\sx_{k-1},\sx_i)$$ does not belong to $\cL_{k,s}$ because the sum of the $k-1$ terms on the right hand side does not depend on $\sx_k$, while the term on the left hand side does. By Lemma~\ref{lawtoword} $e^r$ is a word $w$ in the $e^{\sx_i}$'s and it does not vanish entirely on $G^k$ although it vanishes at the point $(x_1,\ldots,x_{k})$. Hence $D_0$ satisfies the requirements of Lemma \ref{generic} and we are done.
\end{proof}

If $G$ is not Diophantine, then a much stronger statement holds:

\begin{prop}\label{nondioph}Let $G$ be a connected nilpotent Lie group, which is non-Diophantine for $k$-tuples. Then there is a word map $\omega \in F_{k,G} \setminus \{1\}$ such that all $k$-tuples $\bg \in G^k$ such that $\omega(\bg) \neq 1$ are non-Diophantine. In particular there is an open dense subset of $G^k$ made of non-Diophantine $k$-tuples.
\end{prop}

Before giving the proof we make the following observation.

\begin{lem}Let $F$ be a finitely generated subgroup of a connected $s$-step nilpotent Lie group $G$ and $\Gamma\leq F$ a subgroup. If  $\min \{d(x,1)\,;\, x \in \Gamma \cap B_F(n)\setminus\{1\}\}\ll n^{-A}$ for all $A>0$, then $\Gamma$ is non-Diophantine in $G$.
\end{lem}

\begin{proof}As is well-known every subgroup of a finitely generated nilpotent group is finitely generated, so it makes sense to require that $\Gamma$ be non-Diophantine as a subgroup of $G$, or not. We claim that the word metric on $\Gamma$ is bounded above by a fixed power of the trace on $\Gamma$ of the word metric coming from $F$, namely $\ell_\Gamma(\gamma) \leq O(\ell_F(\gamma)^s)$. Clearly this implies the lemma. To see the claim let $M$ be the Zariski closure of $\Gamma$ in the Malcev closure $N$ of $F$. It is a connected and closed subgroup of $N$ (\cite[II.2.5, II.2.6]{raghunathan}). Given a norm $||x||$ on the Lie algebra of $N$, any homogeneous quasi-norm $|\cdot|_N$ on $N$ satisfies $c|x|_N\leq ||x|| \leq C|x|_N^s$ (\cite[Example 2.3]{breuillard}) for some positive constants $c,C>0$ assuming $||x||\geq 1$ say. Also the same holds for a homogeneous quasi-norm $|\cdot|_M$ on $M$ if $x \in M$. The ball-box principle (see e.g. \cite[Proposition 4.5]{breuillard} applied to the distance on $N$ induced by the word metric on $F$ \cite[Example 4.3 (1)]{breuillard}) tells us that $|\cdot|_N$ and $\ell_F$ (resp. $|\cdot|_M$ and $\ell_\Gamma$) are comparable up to multiplicative and additive constants. The claim follows.
\end{proof}

\begin{proof}[Proof of Proposition \ref{nondioph}] Let $N$ be the simply connected Lie group $(\cF_{k,s}/\cL_{k,s}(\fg),*)$. Recall (see Prop. \ref{groupofwordmaps}) that $F_{k,G}$ is a lattice in $(\cF_{k,s}/\cL_{k,s}(\fg)_\QQ,*)$ and that it embeds into $N$ via the natural Lie homomorphism $\rho: (\cF_{k,s}/\cL_{k,s}(\fg)_\QQ,*) \to N$ (associated to the Lie algebra homomorphism $\Lambda$ of $(\ref{thelambdamap})$). By Theorem \ref{diophantinelaws}, $\rho(F_{k,G})$ is a non Diophantine subgroup of $N$. Let $\Gamma \leq F_{k,G}$ be a subgroup of the form $\Gamma=\rho^{-1}(\rho(F_{k,G}) \cap L)$, where $L$ is a connected subgroup of $N$, which is such that $\rho(\Gamma)$ is non-Diophantine in $L$. Choose $\Gamma$ so that $L$ has minimal possible dimension. Pick $\omega \in \Gamma\setminus \{1\}$.

Now let $\bg \in G^k$ be a $k$-tuple such that $\omega(\bg) \neq 1$. The choice of $\bg$ induces a homomorphism $\phi_\bg: F_{k,G} \to G$, which extends by Malcev rigidity \cite[Theorem~2.11]{raghunathan} to a Lie group homomorphism from $(\cF_{k,s}/\cL_{k,s}(\fg)_\QQ,*)$ to $G$, which factors through $N$ via $\rho$. This yields a Lie group homomorphism $\widetilde{\phi_\bg}: N \to G$. Now the main point is the following simple consequence of the above lemma : if a finitely generated non-Diophantine subgroup of a nilpotent Lie group maps to a Diophantine subgroup under a Lie group homomorphism, then its intersection with the kernel is already non-Diophantine. In our case $\rho(\Gamma)$ is non-Diophantine, so its image in $G$ must also be non-Diophantine unless $\ker \widetilde{\phi_\bg} \cap \rho(\Gamma)$ is non-Diophantine. By minimality of $\dim L$, if this happens, then $L \leq \ker \widetilde{\phi_\bg}$ and so $\phi_\bg(\Gamma)=1$. However this is impossible because we assumed that $\omega(\bg) \neq 1$. We conclude that $\phi_\bg(\Gamma)$ is non-Diophantine in $G$ and hence so is the $k$-tuple $\bg$.
\end{proof}

\subsection{Dependence in $k$ of the Diophantine exponent}\label{dependence} We gather here a few additional remarks about the Diophantine exponent $\beta$, prove Theorem \ref{concludeth} and mention some related open problems. First we have the following simple observation.

\begin{remark} Let $G$ be a connected real Lie group $G$. The set of integers $k \geq 1$ such that $G$ is Diophantine on $k$ letters is an interval $[1,k_0]$, where $k_0$ is either a finite integer (in case $G$ is not Diophantine) or $+\infty$ (in case it is). This is clear given that if a $k$-tuple $(s_1,\ldots,s_k)$ in $G$ is Diophantine, then any subtuple is again Diophantine. The following result is Theorem \ref{everykk} from the introduction. It shows that for any integer $k_0\geq 1$, one may construct a nilpotent Lie group $G$ that is Diophantine on $k$ letters if and only if $k\in[1,k_0]$. We thank A. Gamburd for raising the question of the existence of such a Lie group $G$.
\end{remark}

\begin{thm}\label{everyk}
For any integer $k_0\geq 1$, there exists a connected nilpotent Lie group $G$ such that $G$ is Diophantine for words on $k_0$ letters, but non-Diophantine for words in $k_0+1$ letters.
\end{thm}
\begin{proof}
If $k_0=1,2$, we may conclude using Theorem~\ref{nondiophantine}: any connected nilpotent Lie group of nilpotency step 7 (resp. 6) that is non-Diophantine on $2$ (resp. $3$) letters will do.\\
Now assume $k_0\geq 3$. Let $s=k_0+3$. By Corollary~\ref{twotwoone}, the Young diagram with $s$ boxes and $k_0+1$ rows of shape $\lambda:=(2,2,1,\dots,1)$ occurs with multiplicity in $\cF_{k_0+1,s}^{[s]}$. Take $E_1$ and $E_2$ two copies of $E^{\lambda}$ in $\cF_{k_0+1,s}^{[s]}$ defined over $\QQ$, and let
$$\cL = \{ (x,\lambda x)\in E_1 \oplus E_2 ; x\in E_1\},$$
for some Liouville number $\lambda$. Finally, define $G$ to be the connected simply connected nilpotent Lie group with Lie algebra $\cF_{k_0+1,s}/\cL$. The ideal of laws of $G$ over $k_0+1$ letters is just $\cL$, and it is non-Diophantine in $\cF_{k_0+1,s}$. Hence $G$ is non-Diophantine for words on $k_0+1$ letters, by Theorem \ref{diophantinelaws}.\\
We claim that the ideal of laws of $G$ over $k_0$ letters is $\{0\}$, and in particular is Diophantine.  This will show that $G$ is Diophantine for words in $k_0$ letters  by Theorem \ref{diophantinelaws}.  To see the claim, observe that the ideal of laws $\cL_{k_0,s}(Lie(G))$ on $k_0$ letters is homogeneous, so if it is non-zero it must contain a weight vector (for the diagonal action $(\sx_1,\dots,\sx_{k_0}) \to (t_1 \sx_1,\dots,t_{k_0} \sx_{k_0})$), say $\sr:=\sr(\sx_1,\dots,\sx_{k_0})\in \cF_{k_0,s}$, which, being a law of $Lie(G)$, must belong to $\cL$. The weight of $\sr$ is of the form $(u_1,\ldots,u_{k_0})$. However Theorem \ref{dimirr} below shows that the dimension of the subspace of $\cL \simeq E^{\lambda}$ made of vectors of weight $(u_1,\dots,u_{k_0})$ is the number of semi-standard tableaux of shape $\lambda$ having each number $i=1,\dots,k_0$ occurring $u_i$ times. In particular such a tableau has at most $k_0$ distinct entries. But $\lambda$ has $k_0+1$ rows, and entries are strictly increasing in each column of a standard tableau. Hence any semi-standard tableau of shape $\lambda$ must have at least $k_0+1$ distinct entries. This contradiction proves the claim.
\end{proof}

We will prove here two related results. First we show that if $G$ is $s$-step nilpotent and is not Diophantine on $k$ letters for some $k\geq 1$, then it is not Diophantine on $s$ letters already. In other words:

\begin{prop}\label{sversusk} An $s$-step nilpotent Lie group $G$ is Diophantine if and only if it is Diophantine on $s$ letters.
\end{prop}

\begin{proof}Let $k \geq s$. In view of Theorem \ref{diophantinelaws} and the above remark we are left to prove that $\cL_{k,s}$ is Diophantine in $\cF_{k,s}$ if $\cL_{s,s}$ is Diophantine in $\cF_{s,s}$. For each set $B$ of at most $s$ letters amongst $\sx_1,\ldots,\sx_k$ consider the subspace $V_B$ of $\cF_{k,s}$ spanned by the commutators whose set of  letters occurring in them is precisely $B$. The $V_B$'s are in direct sum, span $\cF_{k,s}$, and they decompose further into weight spaces for the diagonal action $(t_1,\ldots,t_k) \cdot c(\sx_1,\ldots,\sx_k):=c(t_1\sx_1,\ldots,t_k\sx_k)$ on $\cF_{k,s}$. The weights occurring in $V_B$ are of the form $(n_1,\ldots,n_k) \in \NN^k$, where $n_i \neq 0$ if and only if $\sx_i \in B$. The fully invariant ideal $\cL_{k,s}$ also decomposes as a direct sum of weight spaces, and $\cL_{k,s} = \oplus_B \cL_{k,s} \cap V_B$. Observe that for each set $B$ of $s$ letters $\oplus_{B' \subset B} V_{B'}$ is isomorphic to $\cF_{s,s}$ and $\oplus_{B' \subset B} \cL_{k,s}  \cap V_{B'}$ to $\cL_{s,s}$. The result is now a direct consequence of the following lemma.
\end{proof}

\begin{lem}Suppose $V$ is a finite dimensional $\QQ$-vector space and $V_i \leq V$ are $\QQ$-subspaces such that $V=\oplus_i V_i$. Let $L\leq V(\RR)$ be a real subspace such that $L=\oplus_i L \cap V_i(\RR)$. Then $L$ is Diophantine in $V(\RR)$ (with exponent $A>0$) if and only if each $L \cap V_i(\RR)$ is Diophantine in $V_i(\RR)$ (with same exponent $A$).
\end{lem}

\begin{proof} This is easily verified since for some choice of norm on $V(\RR)$, $\|v-l\|=\max\{\|v_i-l_i\|\}$.
\end{proof}

It is worthwhile to stress that the argument above shows that if $A=A_s$ is the exponent making $\cL_{s,s}$ Diophantine in $\cF_{s,s}$ (see Definition \ref{diophantine-subspace}), then $\cL_{k,s}$ is again Diophantine with the same exponent as a subspace of $\cF_{k,s}$, for all $k \geq s$. Here is an immediate consequence of this and of the proof of Proposition \ref{diophantinepub}.

\begin{prop}\label{limsupbeta}Let $G$ be a Diophantine $s$-step nilpotent Lie group. Then for every $\beta>s$ there is $k_0\geq 1$ such that $G$ is $\beta$-Diophantine for $k$-tuples for each $k \geq k_0$.
\end{prop}
\begin{proof} Let $A_k$ be the Diophantine exponent of $\cL_{k,s}(\fg)$ in $\cF_{k,s}$, and $\tau_k$ the growth exponent of $F_{k,G}$ (see $\S$ \ref{subsec: word maps} for notation). By Remark~\ref{add}, the group $G$ is $\beta$-Diophantine for $k$-tuples if $\beta>s+\frac{sA_k}{\tau_k}$. We just observed that for any $k\geq s$, we have $A_k\leq A_s$; as the growth exponent $\tau_k$ goes to $+\infty$ as $k \to +\infty$, the result follows.
\end{proof}

Finally in the opposite direction, using a simple pigeonhole argument, we show that $\beta$ cannot be too small.

\begin{prop} Let $G$ be an $s$-step nilpotent Lie group with Lie algebra $\fg$. Let $d_s$ be the dimension of the last step $\fg^{(s)}$. For every $\epsilon>0$ there is $k_0\geq 1$ such that if, for some $k \geq k_0$, $G$ is $\beta$-Diophantine for $k$-tuples, then $\beta>\frac{1}{d_s} - \epsilon$.
\end{prop}

\begin{proof} Let $e_i$ be the rank of the $i$-th successive quotient in the central descending series of $F_{k,G}$. Then since $F_{k,G}$ is a lattice in $\cF_{k,s}/\cL_{k,s}(\fg)_{\QQ}$ (see Prop. \ref{groupofwordmaps}), $e_i$ is the dimension of $\cF_{k,s}^{[i]}/\cL^{[i]}_{k,s}(\fg)_{\QQ}$ for each $i=1,\ldots,s$. But $\cL_{k,s}^{[i]}(\fg)_{\QQ}$ is an $SL_k$-submodule of $\cF_{k,s}$ for the action by linear substitution of the variables. And by Corollary \ref{dimensioncor} each irreducible $SL_k$-module appearing in $\cF_{k}^{[i]}$  has its dimension equal to a polynomial of degree $i$ in $k$, when $k \geq s$. The number of irreducible modules that may appear in $\cF_{k}^{[i]}$ is bounded in terms of $i$ only, so each $e_i$ is bounded above and below by a polynomial of degree $i$ in the variable $k$. With some extra work it can be shown that $e_i$ is a polynomial in $k$ when $k \geq s$, but we will not need that.

Now the Bass-Guivarc'h formula $(\ref{BGf})$ tells us that $\tau_k-se_s$ is a linear combination of the $e_i$, $i<s$. Hence it is bounded above by a polynomial in $k$ of degree at most $s-1$. In particular $\lim_{k \to +\infty} \frac{se_s}{\tau_k}=1$. Now there are roughly $n^{se_s}$ word maps in the word ball of $F_{k,G}$ with radius $n$ that lie in $F_{k,G}^{[s]}$. For a given $k$-tuple the images of these word maps lie in $G^{[s]}$ and are at (left-invariant Riemannian) distance $O(n)$ from the origin. Hence they lie in a part of $G^{[s]}$ of measure $O(n^{sd_s})$, because their norm in $Lie(G)$ is of order $O(n^s)$ (by the ball-box principle, e.g. see \cite[Prop. 4.5]{breuillard}). By pigeonhole there must be a word of length $\leq 2n$ lying at (Riemannian) distance $\ll n^{-s(\frac{e_s}{d_s}-1)}$ from the origin. Letting $k$ tend to $+\infty$ the result follows.
\end{proof}

For a Diophantine $s$-step nilpotent Lie group $G$, we can set
$$\beta_k:=\inf\{\beta>0 ; G \textnormal{ is } \beta\textnormal{-Diophantine for } k\textnormal{-tuples}\}.$$ The above discussion shows that

$$\frac{1}{d_s} \leq \liminf \beta_k \leq \limsup \beta_k \leq s$$
where $d_s=\dim G^{[s]}$. At first glance it may seem surprising that these bounds hold for every Diophantine nilpotent group regardless of the Diophantine exponent $A_k$ of $\cL_{k,s}(\fg)$ in $\cF_{k,s}$.

It seems plausible that $\lim_k \beta_k$ exists and is $>0$ for every Diophantine nilpotent Lie group. This can be verified in certain cases. For example, a Borel-Cantelli argument can be used to prove that the critical exponent for the $3$-dimensional Heisenberg group is $\beta_k:=1-\frac{1}{k}-\frac{2}{k^2}$. In this case for any $\beta<\beta_k$ almost every $k$-tuple is not $\beta$-Diophantine. It would be interesting to compute exactly the critical exponent say for all nilpotent Lie groups defined over $\QQ$. We plan to address some of these issues in a subsequent paper.

\subsection{Speed of equidistribution}

In \cite{breuillard-equid} the second named author proved that finitely generated dense subgroups of connected nilpotent Lie groups are equidistributed. However no rate of convergence was derived. In part because of the use of ergodic theory through Ratner's theorem in the proof. Of course no good error term is to be expected for general dense subgroups. Indeed even when $G=\RR$, the $2$-generated subgroup $\ZZ+\lambda\ZZ$ is equidistributed (Weyl) but no rate can be expected if say $\lambda$ is irrational and yet extremely well approximable by rationals. It seems likely however that Diophantine dense subgroups are equidistributed with some (polynomial) rate (this is true for $\RR$ by a standard Fourier argument), and perhaps this is even true for random subgroups in any nilpotent Lie group, Diophantine or not. This can be compared with the situation in $SU(2)$, where we expect (cf. Sarnak's spectral gap conjecture) that every (as opposed to almost every) $k$-tuple equidistribute with a good rate. In this case, in stark contrast with the nilpotent case, we know by work of the second named author \cite[Corollary 1.11]{breuillard-sl2} that every $k$-tuple satisfies a weak form of the Diophantine property.

\subsection{Almost laws}\label{almostlaws}

In showing that there are non-Diophantine nilpotent and solvable Lie groups, we proved that there are sequences of words $w_n$ on $k$-tuples which are not laws of $G$ but behave like almost laws in that for every compact subset $K$ of $G^k$, $w_n(K) \to 1$ very fast. We informally call such words almost laws. If $G$ is any algebraic group defined over $\QQ$, then picking a non zero rational point it is easy to see that an almost law cannot shrink every fixed compact set faster than exponentially fast (in the length $l(w_n)$ of the word) for a general $G$, and faster than polynomially fast if $G$ is nilpotent. It may be of interest to observe that the well-known shrinking property of commutators near the identity in any Lie group $G$, implies that the sequence of iterated commutators $w_{n+1}=[w_n,w_{n-1}]$ shrink a fixed neighborhood of $1$ at speed $e^{C\sqrt{l(w_n)}}$. Note that in \cite{thom} a construction is given of a sequence of almost laws $w_n$ such that $w_n(G^k) \to 1$ for every compact group $G$. We end with the following question, say for $G=SU(2)$ : can one find a sequence of words $w_n$ such that $w_n(G^k) \subset B(1,e^{-c l(w_n)})$ for some $c>0$ as $l(w_n) \to +\infty$ ?



\appendix
\section{The free Lie algebra viewed as an $SL_k$-module}
\label{app:weightformula}

In this appendix we recall Witt's formula for the dimension of the weight spaces of the free Lie algebra on $k$ generators. Then we use this formula to decompose the free nilpotent Lie algebra on $k$ generators and step $s$ into irreducible $SL_k$-modules for small values of $k$ and $s$. We also determine precisely for which values of $k$ and $s$ this decomposition is multiplicity-free. As we saw in Section \ref{sec:submodules}, this is key to proving that nilpotent Lie groups in step at most $5$ are Diophantine and to build counter-examples in higher step.

\subsection{Witt's formula} Throughout the appendix, $\cF_k$ denotes the free nilpotent Lie algebra on $k$ generators over a field $K$ of characteristic zero (not necessarily $\RR$ as in the rest of the paper). Let $\cF_k^{[s]}$ be the subspace of elements of homogeneous degree $s$ in $\cF_k$.
The group $SL_k:=SL_k(K)$ acts naturally on $\cF_k^{[s]}$ by linear substitution of the free variables. Let $A$ be the diagonal subgroup of $SL_k$.
The representation $\cF_k^{[s]}$ splits as a direct sum of weight spaces:

$$\cF_k^{[s]}=\oplus_{\chi} V(\chi)$$
where $V(\chi)$ is the weight space associated to $\chi \in A^*$. The weights are multiplicative characters on $A$. Given non-negative integers $n_1,\ldots,n_k$, let $\ell(n_1,\ldots,n_k)$ be the dimension of the weight space with weight
\begin{equation}\label{chinot}
\chi_{(n_1,\ldots,n_k)} : a \mapsto \prod_i a_i^{n_i},
\end{equation}
where $a \in A$ is the diagonal matrix $\diag(a_1,\ldots,a_k)$. We will often abbreviate $\chi_{(n_1,\ldots,n_k)}$ simply by $(n_1,\ldots,n_k)$.

In 1937 Witt proved (\cite[Satz 1]{witt}) what is now known as the Poincar\'e-Birkhoff-Witt Theorem about Lie algebras and their universal enveloping algebras. In his third theorem \cite[Satz 3]{witt}, he deduces from it two dimension formulas. The first gives the dimension of $\cF_k^{[s]}$.

\begin{equation}\label{witwit}
\dim \cF_k^{[s]} = \frac{1}{s} \sum_{d|s} \mu(d) k^{\frac{s}{d}},
\end{equation}
where $\mu(d)$ is the M\"obius function. By M\"obius inversion, this formula is equivalent to:

\begin{equation}\label{wittbis}
\sum_{d|s} d \dim \cF_k^{[d]} = k^s
\end{equation}

The second formula of Witt refines the first and gives the dimension of the homogeneous components. It can be stated as follows.

\begin{thm}[Witt's Character Formula for the free Lie algebra]\label{characterformula}  Let $n_1,\ldots,n_k$ be non-negative integers. The dimension $\ell(n_1,\ldots,n_k)$ of the weight space with weight $(n_1,\ldots,n_k)$ for the $SL_k$-action on the subspace $\cF_k^{[s]}$ of commutators of order $s$ in the free Lie algebra on $k$ generators is:
$$\ell(n_1,\ldots,n_k)= \frac{1}{n} \sum_{d | gcd(n_1,\ldots,n_k)} \mu(d) \frac{(\frac{n}{d})!}{(\frac{n_1}{d})! \cdot \ldots \cdot (\frac{n_k}{d})!},$$
where $n=n_1+\ldots+n_k$.
\end{thm}

In the above formula $\mu(d)$ is the M\"obius function. By M\"obius inversion, the formula is readily seen to be equivalent to:
$$\sum_{m | gcd(n_1,\ldots,n_k)} \frac{1}{m} \ell(\frac{n_1}{m},\ldots,\frac{n_k}{m}) = \frac{((\sum_i n_i) -1)!}{n_1! \cdot \ldots \cdot n_k!}$$
for all choices of non-negative integers $n_1,\ldots,n_k$.

For the reader's convenience, we reproduce Witt's proof below following Serre's treatment of Witt's first dimension formula ($\ref{wittbis}$). See \cite[Chapter 4]{serrelalg} and also Hall's book \cite[chapter 11.2]{hall} for a different proof.

\begin{proof}Let $A$ be the free associative $K$-algebra on $k$ generators (i.e. formal linear combinations of \emph{non-commutative} monomials in $k$ letters). Let $a(n_1,\ldots,n_k)$ be the dimension of the subspace of $A$ generated by  non-commutative monomials with the letter $X_i$ occurring $n_i$ times.

By \cite[Theorem 4.2.1]{serrelalg}, the algebra $A$ is isomorphic to the universal enveloping algebra of the free Lie algebra. Pick a basis $\{C_j\}_j$ of the free Lie algebra on $k$ generators $X_1,\ldots,X_k$ made of commutators $C_j$'s. By the Poincar\'e-Birkhoff-Witt Theorem, $A$ has a basis consisting of monomials of the form
$$C^e:=C_{i_1}^{e_{i_1}}\cdot \ldots \cdot C_{i_k}^{e_{i_k}} \textnormal{ with }i_1<i_2<\ldots<i_k$$
We have $deg(C^e)=\sum_j e_{i_j}d_{i_j}$ and $deg_{X_i}(C^e)=\sum_j  e_{i_j}d_{C_{i_j}}(X_i)$, where $d_{C_j}(X_i)$ is the number of occurrences of the letter $X_i$ in the commutator $C_j$.

Formula $(\ref{wittbis})$  will follow by counting $a(n_1,\ldots,n_k)$ in two ways. On the one hand it is clear that

$$a(n_1,\ldots,n_k) = \frac{(\sum_i n_i)!}{n_1! \ldots n_k!}$$

And on the other hand $a(n_1,\ldots,n_k)$ is also the number of families $\{e_j\}$'s such that
$$n_i=\sum_j e_j d_{C_j}(X_i)$$
for each $i=1,\ldots,k$. Therefore $a(n_1,\ldots,n_k)$ is the coefficient of $t_1^{n_1} \cdot \ldots \cdot t_k^{n_k}$ in the formal power series:

$$\prod_j (1 + t_1^{d_{C_j}(X_1)} \cdot \ldots \cdot t_k^{d_{C_j}(X_k)}  + t_1^{2d_{C_j}(X_1)} \cdot \ldots \cdot t_k^{2d_{C_j}(X_k)} + \ldots + t_1^{md_{C_j}(X_1)} \cdot \ldots \cdot t_k^{md_{C_j}(X_k)} + \ldots)$$
which is the same as
$$\prod_j \frac{1}{1-t_1^{d_{C_j}(X_1)} \cdot \ldots \cdot t_k^{d_{C_j}(X_k)} }.$$
Therefore we have the following identity:
$$\prod_j \frac{1}{1-t_1^{d_{C_j}(X_1)} \cdot \ldots \cdot t_k^{d_{C_j}(X_k)} } = \sum_{n_1,\ldots,n_k}  \frac{(\sum_i n_i)!}{n_1! \ldots n_k!} t^{n_1} \cdot \ldots \cdot t_k^{n_k} = \frac{1}{1-(t_1+\ldots + t_k)}.$$

Using $\log \frac{1}{1-u} = \sum_m \frac{1}{m} u^m$, and taking logs we get:

$$\sum_{n_1,\ldots,n_k,m} \ell(n_1,\ldots,n_k) \frac{1}{m} (t_1^{n_1} \cdot \ldots \cdot t_k^{n_k})^m = \sum_m \frac{1}{m} (t_1 + \ldots + t_k)^m$$

Identifying the coefficients of each term we finally obtain the desired formula:

$$\sum_{m | gcd(n_1,\ldots,n_k)} \frac{1}{m} \ell(\frac{n_1}{m},\ldots,\frac{n_k}{m}) = \frac{((\sum_i n_i) -1)!}{n_1! \cdot \ldots \cdot n_k!},$$
which, applying M\"obius inversion, yields
$$\ell(n_1,\ldots,n_k) = \sum_{d | gcd(n_1,\ldots,n_k)} \frac{\mu(d)}{d} \frac{((\sum_i n_i/d) -1)!}{(n_1/d)! \cdot \ldots \cdot (n_k/d)!}= \frac{1}{n}\sum_{d | gcd(n_1,\ldots,n_k)} \mu(d) \frac{(n/d)!}{(n_1/d)! \cdot \ldots \cdot (n_k/d)!}.$$
\end{proof}

\subsection{Multiplicity-free theorem}

Armed with Theorem \ref{characterformula} we are now able to determine the decomposition of $\cF_k^{[s]}$ into simple $SL_k$-modules. For small values of $k$ and $s$ we can give a complete description of the irreducible submodules. And for all values of $k$ and $s$ we determine when it is multiplicity-free.

\begin{thm}[Multiplicity-free]\label{multfree} Let $\cF_k^{[s]}$  be the subspace spanned by commutators of order $s$ in the free Lie algebra on $k$ generators. Then $\cF_k^{[s]}$ is a multiplicity-free $SL_k$-module if and only if $s\leq 5$ when $k\geq 3$, and if and only if $s \leq 6$ when $k=2$.
\end{thm}


Recall that irreducible representations of $SL_k$ are parametrized by their highest weight. If $n_1 \geq \dots \geq n_k \geq 0$ are non-negative integers, we denote by $E^{(n_1,\dots,n_k)}$ the irreducible representation of $SL_k$ with highest weight $(n_1,\ldots,n_k)$. If $n_{i+1}=n_{i+2}=\dots=n_k=0$, we will write $E^{(n_1,\dots,n_i)}$ for $E^{(n_1,\dots,n_i,0,\dots,0)}$.



We also recall Weyl's dimension formula for $E^{(n_1,\ldots,n_k)}$ (\cite[Theorem~6.3]{fulton-harris})

\begin{equation}\label{weyl}
\dim E^{(n_1,\ldots,n_k)}=\prod_{1\leq i<j\leq k}\frac{n_i-n_j+j-i}{j-i}
\end{equation}

From this formula we derive:

\begin{lem} Given $k\geq s \geq 1$ and $n_1\geq \ldots \geq n_s$ non-negative integers such that $s=n_1+\ldots+n_s$, the dimension of the irreducible $SL_k$-module $E^{(n_1,\ldots,n_k)}$ is the value at $k$ of a polynomial of degree $s$ and coefficients in $\QQ$.
\end{lem}

\begin{proof}Split the product in $s$ factors $F_i:=\prod_{j=i+1}^k\frac{n_i-n_j+j-i}{j-i}$ for $i=1,\ldots,s$. Each such $F_i$ is in fact the product of a rational number independent of $k$ (the product of the factors arising when $i+1<j\leq s$ and the denominators of the fraction when $j=s+1,\ldots,s+n_i$) with $(k+1)\cdot\ldots\cdot(k+n_i)$. So each $F_i$ is a polynomial of degree $n_i$ in $k$ with coefficients in $\QQ$. The claim follows.
\end{proof}

Since the only weights $(n_1,\ldots,n_k)$ that can occur in $\cF_{k,s}^{[s]}$ are such that $n_1+\ldots+n_k=s$, we obtain:

\begin{cor}\label{dimensioncor} If $k \geq s$, the dimension of every irreducible $SL_k$-module occurring in $\cF_{k,s}^{[s]}$ is a polynomial of degree $s$ in $k$.
\end{cor}

If $V$ is a finite dimensional $SL_k$-module and $(n_1,\ldots,n_k)$ is its highest weight, then the irreducible $SL_k$-module $E^{(n_1,\ldots,n_k)}$ with highest weight $(n_1,\ldots,n_k)$ occurs as a submodule of $V$. Since the dimensions of the weight spaces of each irreducible $SL_k$-module are known, or at least can be computed, we have a procedure to decompose the original module $V$ into a direct sum of irreducible submodules.

The dimension of the weight space of weight $(u_1,...,u_k)$ in the irreducible representation with highest weight $(n_1,...,n_k)$ can be determined by counting Young tableaux. Recall that to each irreducible $SL_k$-module $E^{(n_1,...,n_k)}$, $n_1 \geq \dots \geq n_k \geq 0$, is associated the Young diagram $\lambda:=(n_1,...,n_k)$, with $n_i$ boxes in the $i$-th row. A semi-standard tableau of shape $\lambda$ is a filling of the Young diagram $\lambda$  with positive integers (one in each box) in such a way that the rows are non-decreasing and the columns strictly increasing.

We have (see e.g. Fulton-Harris \cite[p.224]{fulton-harris}):

\begin{thm}\label{dimirr}
Given a Young diagram $\lambda:=(n_1,\ldots,n_k)$, the dimension of the weight space of weight $(u_1,...,u_k)$ in $E^{\lambda}$  is the number of semi-standard tableaux of shape $\lambda$ filled with integers $i \in \{1,\dots,k\}$ such that each $i$ occurs $u_i$ times.
\end{thm}

We now turn to the proof of Theorem \ref{multfree} and start by analyzing the case when $s \geq 7$.

\begin{lem}\label{baby} If $s \geq 7$ and $k \geq 2$, then $\cF_k^{[s]}$ is not multiplicity-free.
\end{lem}

\begin{proof} If a weight $\chi_{(n_1,\ldots,n_k)}$ occurs non-trivially in $\cF_k^{[s]}$, then $\sum n_i=s$. Using Theorem~\ref{characterformula}, we compute easily  $\ell(s,0,\ldots,0)=0$ and $\ell(s-1,1,0,\ldots,0)=1$. It follows that the highest weight occurring in $\cF_k^{[s]}$ is $(s-1,1,0,\ldots,0)$ and it occurs with multiplicity one.

Again using Theorem~\ref{characterformula} one computes $\ell(s-2,2,0,\ldots,0)$. If $s$ is odd, then this is $\frac{s-1}{2}$, and if $s$ is even, then it is $\frac{s}{2}-1$. As $s \geq 7$, this is at least $3$. However, counting the number of associated Young diagrams, one sees that the dimension of the weight space $(s-2,2,0,\ldots,0)$ in both $E^{(s-1,1)}$ and $E^{(s-2,2)}$  is $1$. Therefore, it must be that $E^{(s-2,2)}$ occurs with multiplicity $\geq 2$ in $\cF_k^{[s]}$.
\end{proof}

The rest of this subsection is devoted to the study of the cases $s=2$ to $s=6$. In each case, we can work out the decomposition of  $\cF_k^{[s]}$ into irreducible modules using Theorem \ref{characterformula}. We discovered a posteriori that this decomposition had been computed already by R. Thrall in 1942 for each $s$ up to $s=10$, see \cite[p. 387,388]{thrall}. \\

\noindent {\bf Case $\mathbf{s=2}$.} Here $\cF_k^{[2]}$ is always irreducible and coincides with $E^{(1,1)}$. Indeed the highest weight of $\cF_k^{[2]}$ is $(1,1,0,\ldots,0)$ and comparing dimensions, we conclude that only $E^{(1,1)}$ occurs as a submodule of $\cF_k^{[2]}$.\\

\noindent {\bf Case $\mathbf{s=3}$.} In this case $\sum_i n_i = 3$, we see that the only possible values for the $n_i$'s are $(1,1,1,0,...,0)$ and $(2,1,0,...,0)$ and arbitrary permutations of such. Using Theorem \ref{characterformula} we compute $\ell(1,1,1,0,\ldots,0)=2$ and $\ell(2,1,0\ldots,0)=1$. Since the highest weight of $\cF_k^{[3]}$ is $(2,1,0,...,0)$, this implies that the irreducible representation $E^{(2,1)}$ with highest weight $(2,1,0,...,0)$ occurs as a sub-representation. However, by the Weyl's dimension formula $(\ref{weyl})$ we have:

$$\dim E^{(2,1)} = \frac{(k-1)k(k+1)}{3}=\frac{1}{3}(k^3-k)$$

Since this coincides with $\dim \cF_k^{[3]}$ by Witt's first formula $(\ref{wittbis})$, we conclude that $\cF_k^{[3]}$ is the irreducible $SL_k$ module:

$$\cF_k^{[3]} = E^{(2,1)}$$

\bigskip

\noindent {\bf Case $\mathbf{s=4}$.} Since $\sum_i n_i = 4$, we see that the only possible values for the $n_i$'s are $(1,1,1,1,0...,0)$ and $(2,1,1,0,...,0)$, $(2,2,0,...,0)$ and $(3,1,0,...,0)$ arbitrary permutations of such. Using Theorem \ref{characterformula} we compute $\ell(1,1,1,1,0,\ldots,0)=6$, $\ell(2,1,1,0\ldots,0)=3$, $\ell(2,2,0,\ldots,0)=1$, and $\ell(3,1,0,\ldots,0)=1$.

The highest weight of $\cF_k^{[4]}$ is $(3,1,0,...,0)$, so $\cF_k^{[4]}$ contains $E^{(3,1)}$. Now $\cF_k^{[4]} \ominus E^{(3,1)}$ has highest weight $(2,1,1,0,...,0)$, so contains $E^{(2,1,1)}$. By the dimension formula:

$$\dim E^{(3,1)} = \frac{1}{8}(k-1)k(k+1)(k+2)=3C_{k+2}^4$$

$$\dim E^{(2,1,1)} = \frac{1}{8}(k-2)(k-1)k(k+1) $$
Since they sum up to $\dim  \cF_k^{[4]}=\frac{1}{4}(k^4 - k^2)$ we conclude that

$$\cF_k^{[4]} = E^{(3,1)} \oplus E^{(2,1,1)}$$

If $k=2$, then the second term is zero, and $\cF_2^{[4]}$ is irreducible. However if $k \geq 3$, then both terms are nonzero and non-isomorphic, so $\cF_k^{[4]}$ is not irreducible and its irreducible submodules have multiplicity $1$.

In fact $E^{(2,1,1)}$ is isomorphic to $\Lambda^2(\Lambda^2(K^k))$, which is the space of \emph{metabelian} brackets, i.e. generated by the $[[X_i,X_j],[X_k,X_l]]$ 's. Those with distinct letters contribute for half of the weight space with weight type $(1,1,1,1,0,...,0)$.

\bigskip

\noindent {\bf Case $\mathbf{s=5}$.}
Since $\sum_i n_i = 5$, the only possible values for the $n_i$'s are $(4,1,0,...,0)$, $(3,2,0,...,0)$, $(3,1,1,0,...,0)$, $(2,2,1,0,...,0)$, $(2,1,1,1,0,...,0)$ and $(1,1,1,1,1,0...,0)$ and arbitrary permutations of such. Using Theorem \ref{characterformula} we obtain: $\ell(4,1)=1$, $\ell(3,2)=2$, $\ell(3,1,1)=4$, $\ell(2,2,1)=6$, $\ell(2,1,1,1)=12$, $ \ell(1,1,1,1,1)=24$.

In order to decompose $\cF_k^{[5]}$ into irreducibles, we first determine the dimensions of the weight spaces of the irreducible representations of $SL_k$ whose highest weight are in each of the above six families of weights. This is done by counting Young diagrams (or by using the sage command $symmetrica.kostka\_tafel(5)$) and we obtain:

\begin{figure}[h]
\begin{center}
\begin{tabular}{|c|c|c|c|c|c|c|}
\hline
\backslashbox{weight}{module}& $E^{(4,1)}$ & $E^{(3,2)}$ & $E^{(3,1,1)}$ & $E^{(2,2,1)}$ & $E^{(2,1,1,1)}$ & $E^{(1,1,1,1,1)}$\\
\hline
A=(4,1) & 1 &  &  &  &  & \\
B=(3,2) & 1 & 1 &  &  &  & \\
C=(3,1,1) & 2 & 1 & 1 &  &  & \\
D=(2,2,1) & 2 & 2 & 1 & 1 &  & \\
E=(2,1,1,1) & 3 & 3 & 3 & 2 & 1 & \\
F=(1,1,1,1,1) & 4 & 5 & 6 & 5 & 4 & 1\\
\hline
\end{tabular}
\end{center}
\caption{Dimensions of weight spaces of positive weight in the irreducible representations of $SL_k$ of norm $5$.}
\end{figure}

Comparing this data with the values for $\ell$ written above, we conclude that $\cF_k^{[5]}$ is the direct sum of one copy of each one of the above six irreducible modules, except $E^{(1,1,1,1,1)}$:

$$\cF_k^{[5]}= E^{(4,1)} \oplus E^{(3,2)}\oplus E^{(3,1,1)} \oplus E^{(2,2,1)}
\oplus E^{(2,1,1,1)}$$
If $k\geq 4$, then all $5$ irreducible submodules are non-zero and pairwise non-isomorphic.
If $k=3$, the last module is zero but the others are pairwise non-isomorphic:

$$\cF_3^{[5]}= E^{(4,1)} \oplus E^{(3,2)}\oplus E^{(3,1,1)} \oplus E^{(2,2,1)}$$
If $k=2$, then the last three modules are zero and we get:

$$\cF_2^{[5]}= E^{(4,1)} \oplus E^{(3,2)}$$
We conclude:

\begin{lem} In step $s=5$ the $SL_k$-module $\cF_k^{[5]}$ is multiplicity-free for all $k\geq 2$.
\end{lem}

\bigskip

\noindent {\bf Case $\mathbf{s=6}$.} Since $\sum_i n_i = 6$, we see that the only possible values for the $n_i$'s are $A:=(5,1,0,\ldots,0)$, $B:=(4,2,0,\ldots,0)$, $C:=(4,1,1,0,\ldots,0)$, $D:=(3,3,0,\ldots,0)$, $E:=(3,2,1,0,\ldots,0)$, $F:=(3,1,1,1,0,\ldots,0)$, $G:=(2,2,2,0,\ldots,0)$, $H:=(2,2,1,1,0,\ldots,0)$, $I:=(2,1,1,1,1,0,\ldots,0)$, and  $J:=(1,1,1,1,1,1,0...,0)$ and arbitrary permutations of such. We can then compute values of $\ell$ by Theorem \ref{characterformula} and  obtain: $\ell(A)=1$, $\ell(B)=2$, $\ell(C):=5$, $\ell(D)=3$, $\ell(E)=10$, $\ell(F)=20$, $\ell(G)=14$, $\ell(H)=30$, $\ell(I)=60$, $\ell(J)=120$.

In order to decompose $\cF_k^{[6]}$ into irreducibles, we first determine the dimensions of the weight spaces of the irreducible representations of $SL_k$ whose highest weight are in each of the above ten families of weights $A$ to $J$. This is done by counting Young diagrams exactly as we did in the step $5$ case.  Doing this counting (or using the sage command $symmetrica.kostka\_tafel(6)$) we obtain the numbers given in Figure \ref{figuredeux}.

\begin{figure}[h] \label{figuredeux}
\begin{center}
\begin{tabular}{|c|c|c|c|c|c|c|c|c|c|c|}
\hline
\backslashbox{weight}{module}& $E^{A}$ & $E^{B}$ & $E^{C}$ & $E^{D}$ & $E^{E}$ & $E^{F}$ & $E^G$ & $E^H$ & $E^I$ & $E^J$\\
\hline
A=(5,1) & 1 &  &  &  &  & & & & &\\
B=(4,2) & 1 & 1 &  &  &  & & & & &\\
C=(4,1,1) & 2 & 1 & 1 &  &  & & & & &\\
D=(3,3) & 1 & 1 & 0 & 1 &  & & & & &\\
E=(3,2,1) & 2 & 2 & 1 & 1 & 1 & & & & &\\
F=(3,1,1,1) & 3 & 3 & 3 & 1 & 2 & 1 & & & &\\
G=(2,2,2) & 2 & 3 & 1 & 1 & 2 & 0 & 1 & & &\\
H=(2,2,1,1) & 3 & 4 & 3 & 2 & 4 & 1 & 1 & 1 & &\\
I=(2,1,1,1,1) & 4 & 6 & 6 & 3 & 8 & 4 & 2 & 3 & 1 & \\
J=(1,1,1,1,1,1) & 5 & 9 & 10 & 5 & 16 & 10 & 5 & 9 & 5 & 1\\
\hline
\end{tabular}
\end{center}
\caption{Dimensions of weight spaces of positive weight in the irreducible representations of $SL_k$ labelled $E^A$ to $E^J$.}
\end{figure}


Comparing this data with the values for $\ell$ written above, we conclude that $\cF_k^{[6]}$ is the direct sum of the following irreducible modules:

\begin{equation*}
\cF_k^{[6]}= E^{(5,1)} \oplus E^{(4,2)} \oplus (E^{(4,1,1)})^2 \oplus E^{(3,3)}
\oplus (E^{(3,2,1)})^3 \oplus E^{(3,1,1,1)} \oplus (E^{(2,2,1,1)})^2 \oplus E^{(2,1,1,1,1)}
\end{equation*}
Note that $E^{G}$ and $E^J$ do not occur as a submodule.
If $k\geq 5$, then all terms are non-zero.



 Note that the submodules corresponding to weights involving more than $k$ variables do not occur. Taking this remark into account we find the following.

If $k=4$, we get
\begin{equation*}
\cF_4^{[6]}= E^{(5,1)} \oplus E^{(4,2)} \oplus (E^{(4,1,1)})^2 \oplus E^{(3,3)}
\oplus (E^{(3,2,1)})^3 \oplus E^{(3,1,1,1)} \oplus (E^{(2,2,1,1)})^2
\end{equation*}
If $k=3$
\begin{equation*}
\cF_3^{[6]}= E^{(5,1)} \oplus E^{(4,2)} \oplus (E^{(4,1,1)})^2 \oplus E^{(3,3)}
\oplus (E^{(3,2,1)})^3
\end{equation*}
If $k=2$
\begin{equation*}
\cF_2^{[6]}= E^{(5,1)} \oplus E^{(4,2)} \oplus E^{(3,3)}
\end{equation*}
We may then conclude:

\begin{lem} In step $s=6$, when $k \geq 3$, then some irreducible $SL_k$-submodule of $\cF_k^{[6]}$ appears with multiplicity $\geq 2$. If $k=2$, then $\cF_2^{[6]}$ is multiplicity-free.
\end{lem}

This concludes the proof of Theorem \ref{multfree}.

\subsection{The free metabelian Lie algebra}

In order to prove Theorems \ref{metabelian} and \ref{thm:nonDiophgrps}, we need to describe the submodule structure of the free metabelian Lie algebra, i.e. the quotient $\cF_k/\cM_k$, where $\cM_k:=D^2(\cF_k)$ is the second term of the derived series of $\cF_k$.

\begin{lem}\label{irreducible}
Let $\cM_k^{[s]}$ be the homogeneous component of degree $s\geq 2$ of $\cM_k$. Then
$$\cF_k^{[s]}/\cM_k^{[s]} \simeq E^{(s-1,1)}.$$
Moreover, $E^{(s-1,1)}$ does not occur in the decomposition of $\cM_k^{[s]}$ into irreducible submodules.
\end{lem}
\begin{proof}
From the set $\{\sx_1<\dots<\sx_k\}$, construct a P. Hall family $P$ for the free Lie algebra $\cF_k$ (see Serre \cite[Part I, Chapter~IV, \S 5]{serrelalg} for definitions). Denote $P_s$ the elements of $P$ of homogeneous degree $s$, and
$$P_s'=\{[\sx_{i_1},[\sx_{i_2},\dots[\sx_{i_{s-1}},\sx_{i_s}]\dots]]\,:\, i_1\geq i_2\geq\dots\geq i_{s-1} ; i_s>i_{s-1}\}.$$
All other elements of $P_s$ belong to $\cM_k$. Hence these elements span $\cF_k^{[s]}$ modulo $\cM_k^{[s]}$. In fact they form a basis of $\cF_k^{[s]}/\cM_k^{[s]}$. Indeed, the $SL_k$-module $\cF_k^{[s]}$ has highest weight $(s-1,1)$ with highest weight vector $[\sx_1,[\sx_1,[...,[\sx_1,\sx_2]...] \in P'_s$. So and so $\cF_k^{[s]}/\cM_k^{[s]}$ contains $E^{(s-1,1)}$ as an irreducible submodule. But one verifies that the cardinality of $P_s'$ is precisely the number of Young tableaux of shape $(s-1,1)$. Hence it is the dimension of $E^{(s-1,1)}$. This proves the first part of the lemma.
Finally we already observed in the proof of Lemma \ref{baby} that $E^{(s-1,1)}$ has multiplicity one in $\cF_k^{[s]}$.
\end{proof}

\begin{cor} The free metabelian Lie algebra on $k$ generators $\cF_k/\cM_k$ is multiplicity-free as an $SL_k$-module.
\end{cor}

\subsection{Klyachko's theorem and the Kraskiewicz-Weyman Formula}
In 1973, Klyachko determined exactly which irreducible modules appear in the homogeneous component $\cF_{k}^{[s]}$ of degree $s$ of the free Lie algebra over $k$ generators $\cF_{k,s}$.

\begin{thm}[Klyachko's theorem \cite{klyachko}]\label{klyach} Let $k \geq 2$, $s \geq 1$. The irreducible $SL_k$-module $E^{\lambda}$ associated to the Young diagram $\lambda$ appears as a submodule of $\cF_{k}^{[s]}$ if and only if $\lambda$ has $s$ boxes, at most $k$ rows and is not one of the following diagrams: a diagram with just one column $\lambda=(1,\dots,1)$, or just one row $\lambda=(s,0,\dots,0)$, or the two diagrams $\lambda=(2,2)$ and $\lambda=(2,2,2)$.
\end{thm}

We refer the reader to \cite{reutenauer} and \cite{kovacs} for two different proofs of Klyachko's theorem. This beautiful result falls short of providing a description of the multiplicities of irreducible $SL_k$-submodules of $\cF_{k}^{[s]}$.  Theorem \ref{KWformula} below does just that.

Recall that if $\lambda=(\lambda_1,\dots,\lambda_r)$ is a Young diagram with total number of boxes $s$, a \emph{standard tableau} of shape $\lambda$ is a filling of $\lambda$ with $\{1,\dots,s\}$ having increasing rows and increasing columns. For a standard tableau $T$ of shape $\lambda$, with total number of boxes $s$, a \emph{descent} is an index $i$ in $\{0,\dots,s-1\}$ such that $i+1$ is located in a lower row than $i$ in $T$. We denote by $D(T) \subset \{1,\dots,s-1\}$ the set of descents of $T$ and define the \emph{major index} of $T$ as
$$\maj(T) = \sum_{i\in D(T)} i.$$
The following theorem is due to Kraskiewicz and Weyman \cite{kraskiewiczweyman} (see also \cite[Corollary 8.10]{reutenauer}).

\begin{thm}[Kraskiewicz-Weyman Formula]\label{KWformula}
Let $k\geq 2$, $s \geq 1$. The multiplicity of the Young diagram $\lambda$ in the decomposition of $\cF_{k}^{[s]}$ into irreducible $SL_k$-submodules is equal to the number of standard tableaux of shape $\lambda$ with major index congruent to $i$ mod $s$, where $i$ is any fixed integer coprime to $s$. In particular this number does not depend on the choice of $i$.
\end{thm}

Although it is not obvious, one can recover Theorem \ref{klyach} from Theorem \ref{KWformula}, see \cite{johnson}.

Theorem \ref{KWformula} allows us to show the following result about multiplicity of a specific Young diagram in the $SL_k$-module $\cF_k^{[s]}$. This is needed in the proof of Theorem \ref{everyk} only.

\begin{cor}\label{twotwoone}
For $s\geq 6$, $k \geq s-2$, the Young diagram $(2,2,1,\dots,1)$ with a total number of $s$ boxes occurs with multiplicity in $\cF_k^{[s]}$.
\end{cor}
\begin{proof}
For $T$ a standard tableau with $s$ boxes we denote by $T^*$ its conjugate, namely the new tableau whose $i$-th column is the $i$-th row of $T$. Note that $T*$ is also standard and that the set of descents of $T^*$ is just the complement in $\{1,\dots,s-1\}$ of the set of descents of $T$. So we have
\begin{equation}\label{transpose}
\maj(T) = \frac{s(s-1)}{2} - \maj(T^*).
\end{equation}
If $s$ is odd, applying the Kraskiewicz-Weyman Formula with $i=1$, we see that it suffices to find two different standard tableaux of shape $(s-2,2)$ with major index congruent to $\frac{s(s-1)}{2}-1=-1$ mod $s$. The major index of a standard tableau $T$ of shape $(s-2,2)$ is equal to $x+y-2$, where $x$ and $y$ are the entries of the lower row of $T$. Taking the two standard tableaux with respective lower rows $(2,s-1)$ and $(3,s-2)$, we get what we want.\\
If $s=2t$ is even, then (\ref{transpose}) and the Kraskiewicz-Weyman Formula with $i=-1$ show that we just have to find two different standard tableaux of shape $(s-2,2)$ with major index congruent to $t+1$ mod $s$. For that, we take the two standard tableaux with lower rows $(2,t+1)$ or $(3,t)$.
\end{proof}

\begin{remark}
Note that when $s$ is not congruent to 2 mod 4, both $1$ and $\frac{s(s-1)}{2}-1$ are prime to $s$, so that by (\ref{transpose}), the multiplicity of a Young diagram $\lambda$ is $\cF_k^{[s]}$ is equal to the multiplicity of the transpose diagram $\lambda^*$.
\end{remark}

\bibliographystyle{alpha}
\bibliography{biblio}	
\end{document}